\documentclass[a4paper,reqno]{amsart}

\usepackage{amsmath,amsthm,amssymb, color,soul,cancel}
\usepackage[all]{xy}

\usepackage{graphicx,psfrag,epsfig, marginnote,subfigure}

\theoremstyle{plain}

\newtheorem*{main}{Main Theorem}
\newtheorem{lemma}{Lemma}
\newtheorem{corollary}{Corollary}
\newtheorem{proposition}{Proposition}

\theoremstyle{definition}

\newtheorem{remark}{Remark}

\newtheorem{example}{Example}

\def\om{\omega}

\def\a{\alpha}
\def\e{\varepsilon}

\def\l{\lambda}
\def\u{{\bar{u}_\l}}
\def\s{\sigma}
\def\vol{{\mathbb V}{\rm ol}}
\def\m{\mu}
\def\L{\Lambda}
\def\R{\mathbb R}
\def\bA{\mathbb A}

\def\cK{\mathcal K}
\def\T{\mathbb T}

\def\Z{\mathbb Z}
\def\cL{\mathcal L}
\def\cO{\mathcal O}
\def\cI{\mathcal I}
\def\cM{\mathcal M}
\def\fM{\mathfrak M}
\def\cN{\mathcal N}
\def\cU{\mathcal U}

\def\cZ{\mathcal Z}

\def\cSt{\widetilde{\Sigma}}
\def\cAt{\widetilde{\mathcal A}}
\def\cMt{\widetilde{\mathcal M}}
\def\cA{\mathcal A}
\def\cM{\mathcal M}
\def\g{\gamma}
\def\G{\Gamma}

\def \r {\rho}
\def\ie {{\it i.e., }}
\def \supp {{\rm supp\,}}

\title{Aubry-Mather Theory for  Conformally Symplectic Systems}

\author{Stefano Mar\`o}
\address{Dipartimento di Matematica, Universit\`a di Pisa, Italy.}
\email{maro@mail.dm.unipi.it}

\author{Alfonso Sorrentino}
\address{Dipartimento di Matematica, Universit\`a degli Studi di Roma ``Tor Vergata'', Rome, Italy.}
\email{sorrentino@mat.uniroma2.it}

\date{\today}

\begin{document}

\begin{abstract}
{In this article we develop an analogue of Aubry-Mather theory for a class of dissipative systems, namely  conformally symplectic systems, and 
prove the existence of interesting invariant sets, which, in analogy to the conservative case, will be called the  Aubry and the Mather sets. Besides describing their structure and their dynamical significance, we shall  analyze  their attracting/repelling properties, as well as their noteworthy role in driving the asymptotic dynamics of the system.}
\end{abstract}

\maketitle


\section{Introduction}

The aim of this paper is to describe the analogue of Aubry-Mather theory for a class of dissipative systems. More specifically, we shall consider   {\it conformally symplectic systems}, namely flows that do not preserve the symplectic structure, but do alter it up to a constant scaling factor (see Section \ref{sec1} for a precise definition). 
These systems appear in many interesting  contexts:  in physics, geometry, celestial mechanics ({\it e.g.}, the spin-orbit model \cite{Cellettibook}), economics (for examples, discounted systems \cite{Benso,IS}), models of transport (see \cite{DM,WL}), etc $\ldots$ In particular, they describe  physical and mechanical systems characterized by a dissipation  proportional to the momentum (or to the velocity), plus the action of a drifting term (see \eqref{eq3}).

The study of invariant Lagrangian submanifolds for these systems, and in particular  the existence  of {\it KAM tori} ({\it i.e.}, invariant Lagrangian tori on which the motion is conjugate to a rotation), have been thoroughly investigated by several authors and by means of varied techniques (see, for instance, \cite{CCD, CCD2, Massetti,Moser67}).

In this article we are mostly focused in understanding what happens  after these invariant Lagrangian submanifolds stop to exist or, more generally, what can be said about the dynamics and the invariant sets of a  dissipative system. 
Although conformally symplectic systems do not certainly cover the whole spectrum of dissipative systems, they definitely provide the  appropriate setting in which this kind of  questions  can be meaningfully addressed.

Inspired by the celebrated {\it Aubry-Mather} and {\it weak KAM} theories for conservative (Hamiltonian) systems (see \cite{Fathibook, Mather91,Sorrentinobook} and references therein), we shall 
investigate the existence of  {\it Aubry-Mather sets} and their dynamical properties ({\it e.g.}, attractivity or repulsivity).
These sets will be constructed by means of variational methods related to the so-called {\it least action principle}.  As a result of their action-minimizing properties, they  enjoy
a  rich structure  and many interesting dynamical features: in some sense, they can be considered as a sort of generalized invariant Lagrangian submanifolds, although not being in general smooth, nor  having the structure of a manifold. For a more precise statement of our results, we refer to the next subsection.

Previous results on Aubry-Mather sets in the dissipative context have been discussed  by Le Calvez \cite{LeCalvez, LeCalvez2, LeCalvezbook} and Casdagli \cite{Casdagli} in the case of twist maps of the annulus. However, the proofs of their results are based on low-dimensional  topological techniques, which makes impossible their extension to a more general setting.

Our work can be considered as a generalization of these results to conformally symplectic flows on any compact manifold.
Although we consider flows and not maps, a discrete version of our  ideas and techniques can be easily implemented, so to recover them.

Finally, let us remark that in the PDE context, related ideas have been recently exploited  to study the vanishing viscosity limit of solutions to the Hamilton-Jacobi equation (see \cite{DFIZ, IS}).\\

\subsection{Setting and statement of the main results}\label{sec1}

Let $M$ be a finite-dimensional  compact and connected smooth manifold, equipped with a smooth Riemannian metric $g$; we shall denote by $d$ the induced Riemannian distance. Let $TM$ and $T^*M$ denote, respectively, the tangent and cotangent bundles. A point of $TM$ will be denoted by $(x,v)$, where $x\in M$ and $v\in T_xM$, and a point of $T^*M$ by $(x,p)$, where $p\in T_x^*M$ is a linear form on the vector space $T_xM$.
With a slight abuse of notation,  we shall denote both canonical projections $\pi: TM \longrightarrow M$ and $\pi: T^*M \longrightarrow M$.
In the same spirit, $\|\cdot\|_x$ will refer to both the norm induced by $g$ on the fibers $T_xM$ and the dual norm on $T^*_xM$.\\
We denote by $\omega=-d\alpha$ the canonical symplectic form on $T^*M$, where $\alpha$ is the Liouville (or tautological) form. Choosing local coordinates $(x_1,\ldots, x_n, p_1,\ldots , p_n)$ on $T^*M$, one has that $\omega= dx\wedge dp := \sum_{i=1}^n dx_i\wedge dp_i$ and $\alpha=pdx := \sum_{i=1}^n p_idx_i$.\\

\noindent A smooth vector field $X$ on $T^*M$ is said to be {\it conformally symplectic} (CS) if there exists $\l \in \R\setminus\{0\}$ such that $$\cL_X \om =  \l \,\om$$ where $\cL_X$ denotes the Lie derivative in the direction of $X$. Clearly, the symplectic case corresponds to the limit case $\lambda =0$.\\
Observe that if $X$ is conformally symplectic, also $-X$ is conformally symplectic and $\cL_{-X} \om =  -\l \,\om$.
Hence, up to a time-inversion of the flow, one can always choose the sign of $\l$. In particular, the case $\l<0$ corresponds to the {\it dissipative} case.\\
Hereafter, we shall consider the case in which
\begin{equation}\label{hypconfsymp}
\cL_{X} \om =  -\l \,\om \qquad  \mbox{for some}\; \l>0,
\end{equation}
and we would be interested in proving the existence of invariant sets and their {\it attracting} properties. Analogously, one could translate these results to the  opposite case  and prove the existence of invariant sets with  {\it repelling} properties.

\begin{remark}
Conformally symplectic vector fields are related to the notion of {\it conformal symplectic structure} on a manifold. Roughly speaking, a local conformal symplectic manifold is equivalent to a symplectic manifold, but the local symplectic structure is only well-defined up to scaling by a constant. This notion was 
first introduced by Vaisman \cite{Vaisman, Vaisman2} and later studied by several authors, for example by Banyaga in \cite{Banyaga}.\\
\end{remark}

\noindent 
Let us start by studying the properties of conformally symplectic vector fields and by deriving the differential equations that govern the motion induced by $X$. Using Cartan's formula and the closeness of $\omega$, one obtains, denoting $i_X$ the inner product, or contraction, with $X$,
$$
\cL_X \om = d(i_X \omega) + i_X(d\omega)  = d(i_X \omega).
$$
Hence, the conformally symplectic condition and the exactness of $\omega=-d\a$ imply
$$d(i_X\omega - \l \alpha) =0$$
that is, the 1-form $i_X\omega - \l \alpha$ is closed. We define the {\it cohomology class}  of $X$ to be the cohomology class of this 1-form; it will be denoted by $[X] \in H^1(M;\R)$. Observe that here (as well as in the following) we are tacitly identifying the de-Rham cohomology groups $H^1(M;\R)$ and $H^1(T^*M;\R)$ by means of the isomorphisms induced by the projection map $\pi:T^*M\longrightarrow M$, which is a homotopy equivalence, and by its homotopy inverse $\iota: M\longrightarrow T^*M$ given by the inclusion of the zero section.\\

\noindent We say that $X$ is {\it exact conformally symplectic} if $[X]=0$. In this case, there exists a smooth function $H: T^*M \longrightarrow \R$ such that
\begin{equation}\label{eq1}
i_X\omega - \l \alpha = dH.
\end{equation}
We call this function an {\it Hamiltonian} associated to $X$. Vice versa, because of the non-degeneracy of $\omega$, to any function $H:T^*M\rightarrow \R$ one can associate a unique vector field $X_H$ that solves (\ref{eq1}). Observe, in fact, that in local coordinates $(x_1,\ldots, x_n,p_1,\ldots, p_n)$, relation (\ref{eq1}) becomes:
$$
-\dot{p} dx + \dot{x} dp - \l p dx = \frac{\partial H}{\partial x}(x,p) dx + \frac{\partial H}{\partial p}(x,p) dp,
$$
and therefore the vector field is given (in local coordinates) by:
\begin{equation}\label{eqmotion}
\left\{
\begin{array}{l}
\dot{x} =  \frac{\partial H}{\partial p}(x,p)\\
\dot{p} = - \frac{\partial H}{\partial x}(x,p) - \l p.
\end{array}
\right.
\end{equation}
We shall denote by $\Phi^t_{H,\l}$ the associated flow and the corresponding exact CS vector field by $X_{H,\l}$ (so to specify both the Hamiltonian and the dissipation).
Observe that when $\lambda=0$, we recover the classical Hamilton's equations.\\

We shall deal with {\it Tonelli exact conformally symplectic} (TECS) vector fields, namely, exact conformally symplectic vector fields that can be generated by a Tonelli Hamiltonian. Recall that a function $H:\,T^*M\longrightarrow \R$ is called a  {\it Tonelli (or optical) Hamiltonian} if:
\begin{itemize}
\item[i)]  $H \in C^2(T^*M)$;
\item[ii)]  $H$ is strictly convex in each fiber in the $C^2$ sense, \ie the second partial 
vertical derivative ${\partial^2 H}/{\partial p^2}(x,p)$ is positive definite,  as a quadratic form, for any $(x,p)\in T^*M$;
\item[iii)] $H$ is superlinear in each fiber, \ie
$$\lim_{\|p\|_x\rightarrow +\infty} \frac{H(x,p)}{\|p\|_x} = + \infty \qquad \mbox{uniformly in }\,x.$$
\end{itemize}
Condition  iii) is equivalent to ask that for every $A\in \R$ there exists $B=B(A)\geq 0$ such that
$$
H(x,p) \geq A\|p\|_x - B \qquad \forall\, (x,p)\in T^*M.
$$
\noindent Using the compactness of $M$, it is possible to check that the property of being Tonelli is independent of the choice of the Riemannian metric $g$.\\

To a Tonelli Hamiltonian we can associate a {\it Lagrangian}  as its {\it Fenchel transform} (or {\it Legendre-Fenchel transform}): 
\begin{eqnarray}\label{lagrangian}
 L:\; TM &\longrightarrow & \R \nonumber\\
(x,v) &\longmapsto & \sup_{p\in T^*_xM} \{\langle p,\,v \rangle_x - H(x,p)\}\, 
\end{eqnarray}
where $\langle \,\cdot,\,\cdot\, \rangle_x$
denotes the canonical pairing between the tangent and cotangent bundles.\\
Since $H$ is a Tonelli Hamiltonian, it is possible to prove that $L$ is finite everywhere (as a consequence of the superlinearity of $H$), superlinear and strictly convex in each fiber (in the $C^2$ sense). 
Moreover, $L$ is also $C^2$. We shall refer to such  a Lagrangian as a {\it Tonelli Lagrangian}.
One can also check that $H$ can be obtained as the Legendre-Fenchel transform of $L$, \ie,
$$H(x,p)= \sup_{v\in T_xM} \{\langle p,\,v \rangle_x - L(x,v)\}.$$

\noindent Observe that the supremum in (\ref{lagrangian}) is actually a maximum an it is attained at $p_{\rm max}=p_{\rm max}(x,v) \in T^*_xM$ such that $\frac{\partial H}{\partial p}(x,p_{\rm max})=v$. In particular, 
$p_{\rm max} = \frac{\partial L}{\partial v}(x,v)$. This defines a map
\begin{eqnarray} \label{legendre}
\cL_L: TM & \longrightarrow& T^*M \nonumber \\
(x,v) &\longmapsto& \left(  x, \frac{\partial L}{\partial v}(x,v)\right)
\end{eqnarray}
which is called the {\it Legendre transform} associated to $L$ (or $H$). It follows from the assumptions on $L$ and $H$, that this map is a global $C^1$ diffeomorphism, whose inverse is given by
\begin{eqnarray} \label{inverselegendre}
\cL_L^{-1}: T^*M & \longrightarrow& TM \nonumber\\
(x,p) &\longmapsto& \left(  x, \frac{\partial H}{\partial p}(x,p)\right).
\end{eqnarray}
{Using the (inverse) Legendre transform $\cL_L^{-1}$  we can transport the flow $\Phi^t_{H,\l} = (x(t), p(t))$ to the tangent bundle $TM$ and define the corresponding Lagrangian flow
 $$
 \Phi^t_{L,\l}=(x(t), v(t)) = \cL_L^{-1}(x(t), p(t)) = \left( x(t), \frac{\partial H}{\partial p}(x(t), p(t))  \right),
 $$
where, using the equation of motion (\ref{eqmotion}), $v(t)=\dot{x}(t)$; see also Section \ref{sec:discount}.\\
\noindent Using the definition of $L$ in (\ref{lagrangian}), one can deduce the following important inequality, known as {\it Legendre-Fenchel inequality}, which will play an important role in our discussion:
\begin{equation} \label{LegFenineq}
L(x,v) + H(x,p)\geq \langle p,\,v \rangle_x 
\end{equation}
for each $x \in M$,  $v\in T_xM$ and $ p\in T^*_xM$. In particular, this inequality becomes an equality if and only if  $(x,p)=\cL_L(x,v)$.\\

\begin{remark}[{\bf Non-exact case}] \label{remnonexact}
Suppose that $X$ is conformally symplectic and it has cohomology class
$[X]=c\in H^1(M;\R)$. Let $\eta_c$ be any smooth closed $1$-form on
$M$ with cohomology class $c$ and see it as a $1$-form on $T^*M$; in
local coordinates it will be represented by $\eta_c = \eta_c (x) \, dx
:= \sum_{i=1}^n \eta_{c,i}(x)dx_i$. Hence, we have
$$i_X\omega - \l \alpha =\eta_c + dH,$$
which in local coordinates will be:
$$
-\dot{p} dx + \dot{x} dp - (\l p  +\eta_c(x)) dx = \frac{\partial H}{\partial x}(x,p) dx + \frac{\partial H}{\partial p}(x,p) dp.
$$
Therefore, the vector field is given (in local coordinates) by:
\begin{equation}\label{eq3}
\left\{
\begin{array}{l}
\dot{x} =  \frac{\partial H}{\partial p}(x,p)\\
\dot{p} = - \frac{\partial H}{\partial x}(x,p) - \l p -\eta_c(x) = - \frac{\partial H}{\partial x}(x,p) - \l (p + \frac{\eta_c(x)}{\lambda}).
\end{array}
\right.
\end{equation}
In order to keep track of all information, we should denote $X$ by $X_{H,\l,c}$ (observe that knowing $H$, $\l$ and $c$, the representative $\eta_c$ is identified univocally); the exact case $X_{H,\l}$ would then correspond to $X_{H,\l,0}$.

\noindent Consider now the change of coordinates $(x, P)=(x, p + \frac{\eta_c(x)}{\lambda}))$, which is symplectic (but not necessarily exact) due to the closedness of $\eta_c$. If we denote 
$\widehat{H}(x,P):= H(x, P-\frac{\eta_c(x)}{\l})$, 
we can see that the equations of motion in these new coordinates become:

\begin{equation*}
\left\{
\begin{array}{l}
\dot{x} =  \frac{\partial \widehat{H}}{\partial P}(x,P)\\
\dot{P} = - \frac{\partial \widehat{H}}{\partial x}(x,P) - \l P.
\end{array}
\right.\\
\end{equation*}

\vspace{10 pt}

\noindent Hence, the non-exact case can be transformed into an exact one, modulo a suitable symplectic change of coordinates (which is of course non-exact). We shall refer to $\widehat{H}$ as the {\it Hamiltonian} of $X$.
Observe that 
\begin{equation}\label{modifiedHam}
H(x,p)= \widehat{H}(x,\frac{\eta_x}{\l}+p).
\end{equation}
This is analogous to Mather's idea, in the conservative case, of changing the Lagrangian (and consequently the Hamiltonian) by subtracting closed $1$-forms.\\
\noindent Finally, let us compute the Lagrangian corresponding to the modified Hamiltonian 
$$H_\theta(x,p)=H(x,p+\theta(x)),$$
where $\theta$ denotes a closed $1$-form on $M$ . It is easy to check, using (\ref{lagrangian}), that: 
 \begin{eqnarray*}
 L_\theta(x,v) &=& \sup_{p\in T^*_xM} \{\langle p,\,v \rangle_x - H_\theta(x,p)\} \\
 &=&   \sup_{p\in T^*_xM} \{\langle p,\,v \rangle_x - H(x,p+\theta(x))\} \\
  &=&   \sup_{p\in T^*_xM} \{\langle p+\theta(x),\,v \rangle_x - H(x,p+\theta(x))\} - \langle \theta(x), v \rangle \\
    &=&   L(x,v) - \langle \theta(x), v \rangle.
 \end{eqnarray*}
 Hence, the Lagrangian changes by a linear term given by the action of the $1$-form $\theta$ on tangent vectors $v$.\\
 \end{remark}

\vspace{10 pt}


In order to state the main results, let us clarify some notions.

\noindent We say that a compact set ${\mathcal S}$ is a global {\it attracting set} for  $\Phi_{H,\l}$, if for each open neighborhood ${\mathcal U} \supset {\mathcal S}$ and for each $(x,p)\in T^*M$, there exists $t_0=t_0(x,p,{\mathcal U})$ 
such that $\Phi_{H,\l}^t(x,p) \in \mathcal U$ for all $t\geq t_0$.
In other words, each orbit, after some time, will get arbitrarily close to $S$ (and remain close thereafter); in particular, this means that 
$S$ contains the $\omega$-limit set\footnote{
We recall that the {\it $\omega$-limit set} of the orbit starting at $(x,p)$, that we denote  by $\Omega_{\infty}(x,p)$,  is defined as the set of points $(\bar x,\bar p)\in T^*M$ for which there exists a sequence $(t_k)$, $t_k\to +\infty$ as $k\to +\infty$ such that 
  \[
  \lim_{k\to +\infty}\Phi_{H,\l}^{t_k}(x,p)=(\bar x,\bar p)
  \]
} of every orbit.
Note that an attracting set is clearly forward-invariant, but it might not be backward-invariant. Hence,
 we say that a compact set ${\mathcal K}$ is a global {\it attractor}, if it is a global attracting set and it is also invariant. 
A global attractor $\mathcal K$ will be said {\it maximal} if it is not properly contained in any other attractor. 
As a part of our analysis, we shall show the existence of a  maximal global attractor for conformally symplectic systems and describe its structure and properties.  

\noindent For more insight on the concept of attractor (and on several other definitions that appear in the literature), see for example, \cite{Conley, milnor1, milnor2}. \\
In the statement of the Main Theorem  we refer to $C^1$ {\it exact Lagrangian graphs}.
The precise definition of Lagrangian submanifolds and their main properties will be discussed in subsection \ref{subsecLaggraph}; here we point out that  a  $C^1$ exact Lagrangian graph can be simply described as  a graph in $T^*M$ of the form $\{(x,du):\; x\in M\}$ where $u: M \longrightarrow \R$ is a $C^2$ function.
  \\
  \medskip
  
  Let us state our Main Theorem. We state it for Tonelli exact conformally symplectic vector fields, but -- as we have pointed out in Remark \ref{remnonexact} -- the results can be easily rephrased for the non-exact case.\\

  \begin{main} 
 Let $M$ be a finite-dimensional compact connected smooth Riemannian manifold without boundary.
 Let $H: T^*M \longrightarrow \R$ be a Tonelli Hamiltonian and $L: TM \longrightarrow \R$ the associated Tonelli Lagrangian. For each $\l>0$, let us consider the exact conformally symplectic vector field $X_{H,\l}$. Then:
      \begin{itemize}
      \item[(i)] There exists the {\em maximal global attractor} $\mathcal K_{H,\l}$ for $X_{H,\l}$.
      \item[(ii)] There exists a non-empty compact invariant set $\cA^*_{H,\l}$, called {\em the Aubry set} for $X_{H,\l}$, with the following properties:
      \begin{itemize}
      \item[a)] The canonical projection $\pi: T^*M \longrightarrow M$ restricted to $\cA^*_{H,\l}$ is a bi-Lipschitz homeomorphisms (Mather's graph property).
      \item[b)] $\cA^*_{H,\l}$ is supported on the graph of the unique (Lipschitz) solution $\u$ of the $\l$-discounted Hamilton-Jacobi equation $\l \u + H(x,d\u)=0$.
      \item[c)] If there exists an invariant $C^1$ exact Lagrangian graph $\Lambda$, then $\cA^*_{H,\l} = \Lambda$.
      \item[d)] The following inclusions hold:
        $$\cA^*_{H,\l} \subseteq \mathcal{K}_{H,\l} \subseteq \{(x,p): \; \l \u(x) + H(x,p) \leq 0\}.$$ 
        In particular, $\cA^*_{H,\l}$ is the maximal compact invariant set contained in $\{(x,p): \; \l \u(x) + H(x,p) = 0\}$.  
      \item[e)] All orbits in $\cA^*_{H,\l}$ are global minimizers for the $\l$-discounted Lagrange action. More specifically, for any $(x,p) \in \cA^*_{H,\l}$ let us denote
        $\g_{(x,p)}(t) := \pi(\Phi_{H,\l}^t(x,p))$. Then, for every continuous piecewise $C^1$ curve $\sigma: [a,b] \longrightarrow M$ such that
        $\s(a)=\g_{(x,p)}(a)$ and $\s(b)=\g_{(x,p)}(b)$, we have:
        $$
        \int_a^b e^{\l t} L(\g_{(x,p)}(t), \dot{\g}_{(x,p)}(t))\,dt \leq \int_a^b e^{\l t} L(\s(t), \dot{\s}(t))\,dt.
        $$
      \end{itemize}
    \item[(iii)]  Let $\fM_{L,\l}$ denote the set of invariant (Borel) probability measures for $\Phi_{L,\l}$.  We say that $\mu \in \fM_{L,\l}$ is {\it action-minimizing} if
      $$
      \int_{TM} (L- \l \u)\, d\mu = \min_{\nu \in \fM_{L,\l}} \int_{TM} (L- \l \u)\, d\nu.
      $$
      Let $\cL_L$ denote the Legendre transform associated to $L$ (see \eqref{legendre}) and let us define the set
      $$
      \cM^*_{H,\l} :=  \cL_L \left(     \overline{\bigcup \{ {\rm supp}\,\m: \; \mu\; \mbox{is action-minimizing} \}}
      \right).$$
      This set, 
      which is called the {\em Mather set} of $X_{H,\l}$, satisfies the following properties:
      \begin{itemize}
      \item[a)] It  is non-empty, compact, invariant and recurrent.
      \item[b)] The restriction  of $\pi$ to $\cM^*_{H,\l}$ is a bi-Lipschitz homeomorphisms (Mather's graph property).
      \item[c)] The following inclusion holds:
        $$
        \cM^*_{H,\l} \subseteq \cA^*_{H,\l}.
        $$
        More specifically:
        $$
        \mu \; \mbox{is action-minimizing}
        \qquad  \Longleftrightarrow \qquad 
        \supp \mu \subseteq \cL_L^{-1}\left( \cA^*_{H,\l} \right).
        $$
        \item[d)] Let $\cM_H^*$ be the Mather set for the corresponding conservative flow $X_{H,0}$. Then, for every neighborhood $\cU \supset \cM^*_H$, the sets $\cM^*_{H,\l}$ are definitely contained in $\cU$ as $\l \rightarrow 0^+.$
      \end{itemize}
      \end{itemize}
  \end{main}

\begin{remark}\label{remIntro}
{\it i}) What we have denoted in the main theorem $\cA_{H,\l}^*$ and $\cM_{H,\l}^*$ correspond to the (inverse) Legendre transform of the sets 
$\cAt_{L,\l}$ and $\cMt_{L,\l}$ that we are going to define in \eqref{defA} and \eqref{defMather}, by means of variational methods. \\
{\it ii}) The inclusions in properties (ii,d) and (iii,c) can be strict; see for instance Example \ref{expendulum} and \ref{exmane} in Section \ref{sec:examples}.\\
{\it iii}) Observe that although the Aubry and the Mather sets are contained in the maximal attractor $\cK_{H,\l}$, 
they might not be attractors themselves (see Example \ref{expendulum} in Section \ref{sec:examples}).\\
{\it iv}) A convergence result similar to property (iii,d) does not hold in general for the Aubry set; see Remark \ref{remconv} for more details.
\end{remark}

\subsection*{ Organization of the article} 
The proofs of the results stated in the Main Theorem are spread throughout the article. In order to help the reader identify them, we list here below more precise references:
\begin{itemize}
\item $\cK_{H,\l}$ is defined in Section \ref{sec:attract}, more precisely in \eqref{defK}; the proof that it is a maximal global attractor is in Proposition \ref{Kattractor}.
\item $\cA^*_{H,\l}$ is defined in Section \ref{sec:aubry}, more precisely in \eqref{defA}. Its properties (a,b,e) are discussed in items a.1-4) after its definition; property (c) follows from 
Proposition \ref{AubryLaggraph}, while property (d) from Proposition \ref{Kattractor}.
\item $\cM^*_{H,\l}$ is defined in Section \ref{sec:mather}, more precisely in \eqref{defMather}. Its properties (a,b,c) are discussed in items m.1-3) after its definition; property (d) is proved in  Corollary \ref{corolconvMat} and Proposition \ref{convMathmeas}.
\end{itemize}
As far as the rest of the paper is concerned, in Section \ref{sec:discount} we introduce the $\l$-discounted Action Functional and the $\l$-discounted Hamilton-Jacobi equations proving some of their properties. Section \ref{sec:weak} is dedicated to  describing the extension of Weak KAM theory to the conformally symplectic case.  The Aubry set is introduced and studied in Section \ref{sec:aubry}. In Section \ref{sec:attract} we investigate asymptotic properties of the flow and use them to construct the global maximal attractor and to study its properties. Action-minimizing probability measures  studied in Section \ref{sec:mather} and used to define the Mather set.
In Section \ref{sec:conv} we address the question of what happens to these objects in the limit from the dissipative to the conservative case.
Finally,  we conclude by describing some illustrative examples in Section \ref{sec:examples}.

\subsection*{Acknowledgments} 
We would like to express our gratitude to Alessandra Celletti for her interest in this work and for many fruitful discussions.
SM acknowledges the support of Marie Curie Initial Training Network {\it Stardust}, FP7-PEOPLE-2012-ITN, Grant Agreement 317185.
AS has been partially supported by the PRIN-2012-74FYK7 grant ``{\it Variational and perturbative aspects of nonlinear differential problem}''.
The authors are grateful to the anonymous referees for their valuable comments and suggestions.

\vspace{20 pt}

 \section{Discounted action and Discounted Hamilton-Jacobi Equations.}\label{sec:discount}
In this section we are going to describe the analog in the conformally symplectic case of two well-known facts in the conservative (Hamiltonian) framework. Namely the correspondence between the Hamiltonian and the Lagrangian flux and the characterization of Lagrangian invariant manifold through solutions of the Hamilton-Jacobi equation. \\

Hereafter we shall consider $X=X_{H,\l}$ to be an exact  CS vector field {as in (\ref{eq1})}, where $\l>0$, $H:T^*M\longrightarrow \R$ is a Tonelli Hamiltonian and $L: TM\longrightarrow \R$ the associated Tonelli Lagrangian.\\

\subsection{Discounted Euler-Lagrange equations and Action}

\noindent Let us start proving that the orbits of the Lagrangian flow $\Phi_{L,\l}$ correspond to solutions of the following Euler-Lagrange equation:

\begin{equation}\label{ELeq}
\left\{\begin{array}{l}
\dot{x} = v\\
 \frac{d}{dt} \left( e^{\l t} \frac{\partial L}{\partial v} \right) = e^{\l t} \frac{\partial L}{ \partial x}.
\end{array}\right.
\end{equation}

\bigskip

\noindent More precisely, the following holds.\\

\begin{proposition}\label{propELHam}
If $(x(t), p(t))$ is a solution of (\ref{eqmotion}), then $(x(t), v(t)) = \cL_L^{-1}(x(t), p(t))$ is a solution of (\ref{ELeq}). Conversely, if $(x(t), v(t))$ is a solution of (\ref{ELeq}), then $(x(t), p(t)) = \cL_L(x(t), v(t))$ is a solution of (\ref{eqmotion}).
\end{proposition}
\smallskip

\begin{proof}
Let us work in a coordinate chart. Using the definitions of $\cL_L$ and $\cL_L^{-1}$, one can check that (see, for instance, \cite[Proposition 2.6.3]{Fathibook})
$$\dot{x}(t)=\frac{\partial H}{\partial p}(x(t),p(t))=v(t),$$ 
which provides the first equation in \eqref{ELeq}, and 
$$ \frac{\partial L}{ \partial x} (x(t),v(t)) = - \frac{\partial H}{ \partial x} \left(x(t), \frac{\partial L}{\partial v}(x(t),v(t))\right) =  - \frac{\partial H}{ \partial x} (x(t), p(t)).$$
Therefore:
{\small
\begin{eqnarray*}
\dot{p}(t)= -\frac{\partial H}{\partial x}(x(t),p(t)) - \l p(t) 
&\Longleftrightarrow&  \frac{d}{dt} \left( \frac{\partial L}{\partial v}(x(t), v(t))\right)  = \frac{\partial L}{\partial x}(x(t), v(t)) - \l \frac{\partial L}{\partial v}(x(t), v(t)) \\
&\Longleftrightarrow&
\frac{d}{dt} \left( e^{\l t} \frac{\partial L}{\partial v}(x(t), v(t))\right)  = e^{\l t}\frac{\partial L}{\partial x}(x(t), v(t)),
\end{eqnarray*}}
which proves that the second equation in \eqref{ELeq} is also solved.
The converse statement can be proved similarly.
\end{proof}

\vspace{20 pt}

\noindent Solutions of (\ref{ELeq}) have a variational characterization. Let  $\g:[a,b]\longrightarrow M$ be a continuous piecewise $C^1$ curve with $-\infty < a <b < +\infty$. We define its {\it discounted action} to be 
$$
 \bA_{L,\l} (\g) = \int_a^b e^{\l t} L(\g(t), \dot{\g}(t))\,dt.
$$
It is a classical result in the calculus of variations, that solutions to (\ref{ELeq}) are in 1-1 correspondence with  $C^2$ extremal curves of the  {discounted action functional}, for the fixed-end problem.  {More precisely we say that a $C^2$ curve $\g:[a,b]\longrightarrow M$ is an {\it extremal} of $\bA_{L,\l}$ 
for the fixed endpoint problem (recall that we are assuming $L$ to be at least $C^2$) if
for every $C^2$ variation
$$
\G: [a,b]\times [-\e,\e] \longrightarrow M
$$
({\it i.e.}, $\G(t,0)=\g(t)$ for all $t\in [a,b]$ and  $\G(t,s)=\g(t)$ in a neighborhood of $(a,0)$ and $(b,0)$), we have
\begin{equation}\label{firstvariation}
\frac{d}{ds} \left(\int e^{\l t} L(\G(t,s),\partial_t \G(t,s)) \,dt\right)_{| s=0} = 0.
\end{equation}
Observe  that if $\g: [a,b]\longrightarrow M$ is a $C^2$ extremal of $\bA_{L,\l}$, then for every $[a',b']\subset [a,b]$ the restriction $\g|[a',b']$ is still a $C^2$ extremal; hence, reducing to a coordinate chart, one can show that if \eqref{firstvariation} is satisfied for all $C^2$ variations $\G$, then  $\g$  satisfies  (in local coordinates) \eqref{ELeq}.
For a more detailed discussion of these resuts, we refer the reader,  for example, to \cite[Chapter 3 \S 12]{Arnold} or \cite[Section 2.1]{Fathibook}. }
\\

\begin{remark} \label{defactionmin}
In the following we shall be interested in {\it (discounted) action-minimizing curves}. We say that a continuous piecewise $C^1$ curve 
$\g:[a,b]\longrightarrow M$ minimizes the discounted action if
$$
 \bA_{L,\l} (\g) \leq  \bA_{L,\l} (\sigma),
$$
for every continuous piecewise $C^1$ curve $\sigma:[a,b]\longrightarrow M$ such that $\sigma(a)=\gamma(a)$ and $\sigma(b)=\gamma(b)$. 
{Proceeding, for example, as in \cite[Proposition 2.3.7 and Corollary 2.2.12]{Fathibook}), it is possible to show   that $\g$ is a $C^2$ extremal of $\bA_{L,\l}$ and satisfies \eqref{ELeq}.}

Analogously, a continuous piecewise $C^1$  curve $\g: I \longrightarrow M$, where $I$ is an unbounded interval, is said to minimize the discounted action, if 
for every compact subinterval $[a,b]\subset I$, $\g|[a,b]$ is action-minimizing.
\end{remark}

\medskip

\subsection{Invariant Lagrangian graphs and Discounted Hamilton-Jacobi equation}\label{subsecLaggraph}
 Let us consider $\L$ to be a $C^1$ Lagrangian submanifold of dimension $n$ in $T^*M$, namely, if we denote by $i_\L: \L \longrightarrow T^*M$ the inclusion, we have that
 $i_\L^*\omega \equiv 0$, \ie the symplectic form vanishes when restricted to the tangent bundle of $\L$. \\
We are interested in $C^1$ Lagrangian graphs over the zero-section of $T^*M$. Recall that smooth Lagrangian graphs  correspond to the  graph of closed $1$-forms on $M$; we shall refer to the cohomology class of $\L\in H^1(M;\R)$ as  the cohomology class of the corresponding $1$-form (one can give a more intrinsic definitions, which extends to more general Lagrangian submanifolds). 

We suppose that $\L$ is invariant under $\Phi_{H,\l}$ and exact in the sense that its cohomology class is 0;
moreover, let $u:M\longrightarrow \R$ be a $C^2$ function such that $ \L = {\rm Graph}(du)$. 
The invariance property translates into the fact that the vector field $X_{H,\l}$ is tangent to $\Lambda$ at any point.

In the conservative case the invariance of a Lagrangian graph can be characterized in terms of being solution of a PDE, known as Hamilton-Jacobi equation. We would like to 
discuss the analogue of this characterization in the dissipative (CS) case.

Let us consider the function ${F}(x,p)= \l u(x) + H(x, p)$. Observe that $ dF(x,p) = \l du(x) + dH(x,p)$. Let $W=(W_x, W_p) \in T_{(x,p)} \L$ be any tangent vector to $\L$ at a point $(x,p)=(x,du(x))$. 
Using the definition of $H$ in (\ref{eq1}), the definition of $\alpha$, the fact that $X_H$ is tangent to $\Lambda$ and the fact that $\Lambda$ is Lagrangian (so it vanishes when applied to tangent vectors to $\Lambda$) we obtain:
\begin{eqnarray*}
\langle dF(x,p), W \rangle &=&
 \l \langle du(x), W_x \rangle + \langle dH(x,p), W \rangle \\
&=&  \l \langle du(x), W_x \rangle +  i_{X_{H,\l}(x,p)} \omega (W) - \l \langle \alpha(x,p), W\rangle\\
&=& \l \langle du(x), W_x \rangle - \l \langle du(x), W_x\rangle =0.
\end{eqnarray*}
It follows that $F$ is constant on $\Lambda$, which gives the {\it $\l$-discounted Hamilton-Jacobi Equation} (or simply, when there is no risk of ambiguity, {\it discounted Hamilton-Jacobi Equation})
\begin{equation}\label{eqHJ}
\l u(x) + H(x, du(x)) = c \qquad \forall\; x\in M
\end{equation}
for some $c\in \R$.\\

\begin{remark}\label{remarkconstant}
{\it i}) Observe that if $u$ is a solution to the equation (\ref{eqHJ}), then $v= u+k$ satisfies $\l v(x) + H(x, dv(x)) = c+\l k$. 
Therefore, the constant does not play -- at least in this context -- any important role. Without any loss of generality, it can be assumed to be equal to 0. \\
{\it ii}) Once the constant on the right-hand side is fixed, equation (\ref{eqHJ}) admits at most one smooth solution. This follows from the comparison principle in \cite[Th\'eor\`eme 2.4]{Barles}). 
Actually, under our assumptions on $H$, this equation admits exactly one viscosity solution, as proved in \cite[Theorem 2.5]{DFIZ}. \\
{\it iii}) The exactness condition on the Lagrangian graph $\Lambda$ is essential in order to define the function $F$ (we need a primitive of the representing $1$-form).
\end{remark}
  
\noindent Conversely, if $u\in C^2(M)$ is a solution to the equation (\ref{eqHJ}), then $\L={\rm Graph}(du)$ is an invariant exact Lagrangian submanifold.  Clearly it is Lagrangian and exact, being the graph of an exact $1$-form.
Moreover, if we consider the restriction  $F_\L:= i_\L^*F=F\circ i_\Lambda \equiv 0$ and use (\ref{eq1}), we obtain
\begin{eqnarray*}
0 &=& dF_\L(x) = i_{\L}^*dF = \l i_{\L}^*du + i_{\L}^*dH \\
&=& \l du + i_{\L}^*( i_{X_{H,\l}}\omega - \l\a ) = \l du + i_{\L}^*( i_{X_{H,\l}}\omega) - \l du\\
&=& i_{\L}^*( i_{X_{H,\l}}\omega).
\end{eqnarray*}
It follows from this and the fact that the tangent spaces to a Lagrangian submanifold are maximally isotropic (recall that a vector subspace is  isotropic  if the symplectic form is identically zero when restricted on it), that $X_{H,\l}$ must belong to $T\Lambda$ and therefore, $\L$ is invariant under the flow.\\ 
Summarizing, we have proved
 
 \begin{proposition}\label{invariantlagrangiangraph}
 Let $u:M\longrightarrow \R$ be a $C^2$ function. The exact Lagrangian graph $\L=\{(x,du(x):\; x\in M\}$ is invariant  under $\Phi_{H,\l}$ if and only if 
$ \l u(x) + H(x, du(x)) \equiv c$ for some $c\in \R$.\\
 \end{proposition}
 
 \vspace{5 pt}
 

\section{Action-minimizing orbits and weak KAM theory for Conformally symplectic systems}\label{sec:weak}
As we have seen in the previous section, the existence of exact Lagrangian graphs is related to the existence of   solution (in the classical sense) to the discounted Hamilton-Jacobi equation (\ref{eqHJ}). However, typically -- think, for example, of systems which are not ``close'' to an integrable one -- these solutions (and hence invariant Lagrangian graphs) are very unlikely to exist.\\
Inspired by what happens in the conservative case, we would like to investigate what can be said about the dynamics of a general conformally symplectic Tonelli system. More specifically, whether it is possible to identify invariant sets that --  in some sense to be better specified -- can be considered as generalization of invariant Lagrangian graphs. In the conservative case these sets are what are generally called {\it Aubry-Mather} sets; see for example, just to mention a few references, \cite{Fathibook, Mather91, Sorrentinobook}.

\noindent In analogy to the classical Aubry-Mather and weak KAM theories, the key idea in what follows is to study  action-minimizing properties of the  system (either with reference to orbits or to invariant probability measures) and to use these to introduce a suitable notion of  weak solution to the discounted Hamilton-Jacobi equation. See also \cite{Fathibook, IS, DFIZ}.\\

Let us start by defining a continuous analogue of subsolutions of the discounted Hamilton-Jacobi equation (\ref{eqHJ}).\\
We say that a function $u \in C(M)$ is {\it $\l$-dominated} by $L$ (and denote it by $u\prec_\l L$) if for every $a<b$ and every continuous piecewise $C^1$ curve $\g: [a, b] \longrightarrow M$ one has
\begin{equation}\label{dominated}
e^{\l b} u(\g(b)) - e^{\l a} u(\g(a)) \leq \int_a^b e^{\l t} L(\g(t), \dot{\g}(t))\,dt.
\end{equation}

\noindent Observe that this definition is independent on additive constants, in the sense that if $u\prec_\l L$, then $u+\frac{k}{\l} \prec_\l L+k$ for any $k\in \R$.\\
In the $C^1$ case, $\l$-dominated functions are subsolutions of the discounted Hamilton-Jacobi equation and vice versa. Actually, suppose that $u\in C^1(M)$  satisfies $\l u + H(x, du(x)) \leq 0$ for all $x\in M$ and let $\g:[a,b]\longrightarrow M$ be a continuous piecewise $C^1$ curve. Using the Legendre-Fenchel inequality:
\begin{eqnarray*}
&& e^{\l b} u(\g(b)) - e^{\l a} u(\g(a)) \ = \ \int_a^b \frac{d}{dt}\left(e^{\l t} u(\g(t))\right)\, dt \\
&& \qquad \quad=\  \int_a^b e^{\l t} \left( \l u(\g(t))+ \langle du(\g(t)), \dot{\g}(t)\rangle \right)\, dt \\
&& \qquad \quad \leq\ \int_a^b e^{\l t} \left( \l u(\g(t))+ H(\g(t), du(\g(t))) + L(\g(t), \dot{\g}(t))
 \right) dt \\
&& \qquad \quad \leq\ \int_a^b e^{\l t} L(\g(t), \dot{\g}(t))\,dt,
\end{eqnarray*}
so that $u\prec_\l L$. Vice versa, if $u\prec_\l L$ and $u\in C^1(M)$, then  $\l u + H(x, du(x)) \leq 0$ for all $x\in M$. \\
More generally, if $u$ is only continuous, we have (see also \cite[Proposition 4.2.2]{Fathibook}):
\begin{proposition}\label{propsubsol}
Let $u\prec_\l L$ and assume that $du(x_0)$ exists for some $x_0\in M$. Then
$$\l u (x_0) + H(x_0, du(x_0)) \leq 0.$$
\end{proposition}

\begin{proof}
Let $v\in T_{x_0}M$ and let $\g:[0,1]\longrightarrow M$ be a $C^1$ curve such that $\g(0)=x_0$ and $\dot{\g}(0)=v$. Since $u\prec_\l L$, we have for all $t\in (0,1]$
\begin{eqnarray*}
\frac{e^{\l t} u(\g(t)) - u(x_0)}{t} \leq \frac{1}{t}\int_0^t e^{\l s} L(\g(s), \dot{\g}(s))\,ds.
\end{eqnarray*}
If we let $t\rightarrow 0^+$, we obtain 
\begin{eqnarray*}
L(x_0, v) &\geq& \frac{d}{dt}\left(e^{\l t} u(\g(t))\right)_{\big|t=0} = \l u(x_0) + \langle du(x_0), v\rangle, \\
\end{eqnarray*}
and hence, for all $v \in T_{x_0}M$,
$$
 \l u(x_0) + \langle du(x_0), v\rangle - L(x_0, v) \leq  0. 
$$
From this inequality, we conclude that 
\begin{eqnarray*}
\l u (x_0) + H(x_0, du(x_0))  &=&  \l u (x_0) + \sup_{v\in T_{x_0}M} \left(
\langle du(x_0), v\rangle - L(x_0, v)  \right)\\
&=& \sup_{v\in T_{x_0}M} \left(
\langle \l u (x_0) + du(x_0), v\rangle - L(x_0, v)  \right) \leq 0
\end{eqnarray*}

\end{proof}

We define now a continuous analogue of classical solutions of the discounted Hamilton-Jacobi equation (\ref{eqHJ}).\\
Let $u\prec_\l L$. We say that a curve $\g:[a,b]\longrightarrow M$ is {\it $(u,L)$-$\l$-calibrated} (or simply {\it $(u,L)$-calibrated}, when there is no risk of ambiguity) if 
$$
e^{\l b} u(\g(b)) - e^{\l a} u(\g(a)) = \int_a^b e^{\l t} L(\g(t), \dot{\g}(t))\,dt.\\
$$
Before describing the properties of calibrated curves let us point out the following 
\begin{remark}\label{remarkcalibrated}
We have
\begin{itemize}
\item[({\it i})] The definition of calibrated curve continues to make sense if $a=-\infty$. In fact,  since the manifold $M$ is compact and $u$ is continuous on $M$, hence bounded, then 
$e^{\l a} u(\g(a)) \longrightarrow 0$ as $a\rightarrow -\infty$.
\item[({\it ii})] It is easy to check (for example by contradiction), that if $\g:[a,b]\longrightarrow M$ is {$(u,L)$-calibrated} and $[a',b'] \subset [a,b]$, then $\g_{\big| [a',b']}$ is still $(u,L)$-calibrated.
\item[({\it iii})] The condition of being a calibrated curve is clearly invariant under time-translations. Namely, it is a straightforward check that if $\g:[a,b]\longrightarrow M$ is $(u,L)$-calibrated, then $\g_T: [a-T,b-T]\longrightarrow M$ defined by $\g_T(s)=\g(T+s)$ is still $(u,L)$-calibrated.
\item[({\it iv})] It follows from the definition and the fact that $u\prec_\l L$, that if $\g:[a,b]\longrightarrow M$ is {$(u,L)$-calibrated}, then it minimizes the discounted Lagrangian action $\bA_{L,\l}$
among all continuous piecewise $C^1$ curves $\sigma: [a,b] \longrightarrow M$ such that $\sigma(a)=\g(a)$ and $\sigma(b)=\gamma(b)$. In fact:
\begin{eqnarray*}
\bA_{L,\l}(\g) &=& \int_a^b e^{\l t} L(\g(t), \dot{\g}(t))\,dt = e^{\l b} u(\g(b)) - e^{\l a} u(\g(a))\\
&=& e^{\l b} u(\s(b)) - e^{\l a} u(\s(a)) \leq \int_a^b e^{\l t} L(\s(t), \dot{\s}(t))\,dt =\bA_{L,\l}(\s).
\end{eqnarray*}
In particular, if $a=-\infty$, then $\g$ minimizes the action among all curves defined on $(-\infty, b]$  that ends at $\g(b)$ at time $t=b$ (see Remark \ref{defactionmin}).
Hence, $(\g,\dot{\g})$ is a solution of \eqref{ELeq}; in particular, proceeding as in \cite[Corollary 2.2.12]{Fathibook}), it follows that $\g$ is $C^2$. Actually, using Proposition \ref{propELHam} and equation (\ref{eq3}) one can prove that $\g$ is as smooth as the Lagrangian $L$ {(see also Remark \ref{defactionmin})}.\\
 \end{itemize}
\end{remark}

\noindent For classical solutions we can prove the following property.

\begin{proposition}\label{existencesolutioncase}
Let $u\in C^2(M)$ satisfy $\l u(x) + H(x, du(x)) =0 $ for every $x\in M$, and let $\g(t) = \pi \left( \Phi_{H,\l}^t(x_0, du(x_0)) \right)$ for some given $x_0 \in M$. Then  $\g$ is $(u,L)$- calibrated on every time-interval $[a,b]$.
\end{proposition}

\begin{proof}
We have already proved (see Proposition \ref{invariantlagrangiangraph}), that the corresponding Graph$(du)$ is invariant under $\Phi_{H,\l}$. The invariance condition implies that
$$
\frac{\partial L}{\partial v}(\g(t), \dot{\g}(t)) = du(\g(t)) \quad \forall \; t \in \R,
$$
hence we have equality in the corresponding Legendre-Fenchel inequality. Using this and the fact that $u$ solves (\ref{eqHJ}), we obtain:
\begin{eqnarray*}
&& e^{\l b} u(\g(b)) - e^{\l a} u(\g(a)) \ = \ \int_a^b \frac{d}{dt}\left(e^{\l t} u(\g(t))\right)\, dt \\
&& \qquad \quad=\  \int_a^b e^{\l t} \left( \l u(\g(t))+ \langle du(\g(t)), \dot{\g}(t)\rangle \right)\, dt \\
&& \qquad \quad = \ \int_a^b e^{\l t} \left( \l u(\g(t))+ H(\g(t), du(\g(t))) + L(\g(t), \dot{\g}(t))
 \right) dt \\
&& \qquad \quad = \ \int_a^b e^{\l t} L(\g(t), \dot{\g}(t))\,dt.
\end{eqnarray*}
\end{proof}

\vspace{10 pt}

A first step in order to provide a meaningful notion of weak solutions, is the following observation (see also \cite[Theorem 4.3.8]{Fathibook} for its analogue in the conservative case).

\begin{proposition}\label{propdiffcal1}
Let $u\prec_\l L$ and $\g:[a,b]\longrightarrow M$ a $(u,L)$-calibrated curve.
If for some $t_0\in[a,b]$, the derivative of $u$ at $\g(t)$ exists, then
$$
du(\g(t_0)) = \frac{\partial L}{\partial v} (\g(t_0), \dot{\g}(t_0)) \quad \mbox{and}\quad {\l u(\g(t_0))} + H(\g(t_0),du(\g(t_0)))=0.
$$
\end{proposition}

\begin{proof}
Let assume that $a\leq t_0<b$ (similarly, one show it for $t_0=b$) and consider $h>0$ such that  $t_0<t_0+h<b$. Because of the calibration condition, we have
$$\frac{e^{\l (t_0+h)} u(\g(t_0+h)) - e^{\l t_0}u(\g(t_0))}{h} = \frac{1}{h}\int_{t_0}^{t_0+h} e^{\l s} L(\g(s), \dot{\g}(s))\,ds.$$
If we let $h\rightarrow 0^+$, we obtain 
\begin{eqnarray}\label{passaggio1}
e^{\l t_0}L(\g(t_0), \dot{\g}(t_0)) &=& \frac{d}{dt}\left(e^{\l t} u(\g(t))\right)_{\big|t=t_0} \nonumber \\
&=& e^{\l t_0} \l u(\g(t_0)) + e^{\l t_0} \langle du(\g(t_0)), \dot{\g}(t_0) \rangle \\
&\leq& e^{\l t_0} \big[\l u(\g(t_0)) + H(\g(t_0), du({\g}(t_0)))+ L(\g(t_0), \dot{\g}(t_0)) \big].\nonumber
\end{eqnarray}
Hence:
$$
\l u(\g(t_0)) + H(\g(t_0), du({\g}(t_0))) \geq 0
$$
and using Proposition \ref{propsubsol} we can conclude that 
$$\l u(\g(t_0)) + H(\g(t_0), du({\g}(t_0))) = 0.$$\\
Substituting this in (\ref{passaggio1}) and simplifying  the common factor $e^{\l t_0}$ we get:
$$
\langle du(\g(t_0)), \dot{\g}(t_0) \rangle =  L(\g(t_0), \dot{\g}(t_0)) + H(\g(t_0), du({\g}(t_0))).
$$
Therefore, the Legendre-Fenchel inequality is in this case an equality, and for what we have already recalled:
$$
(\g(t_0),  du(\g(t_0))) = \cL_L(\g(t_0), \dot{\g}(t_0)) \qquad \Longleftrightarrow \qquad du(\g(t_0)) = \frac{\partial L}{\partial v} (\g(t_0), \dot{\g}(t_0)).
$$
\end{proof}

\vspace{10 pt}

It follows that it is important to detect which points of the image of a calibrated curves are points of differentiability of a $\l$-dominated function. 
Observe that in general these functions are not differentiable everywhere (although, being Lipschitz, they are differentiable almost everywhere\footnote{Recall that on a smooth Riemannian manifold $V$ (in our case, it will be either $V=M$, $TM$ or $T^*M$) we say that a set $Z$ has measure zero if  
it has measure zero for the Riemannian volume measure associated to the Riemannian metric. 
In particular, for every coordinate chart  $\psi:  U \subset V \longrightarrow \R^k$, the image $\psi(U\cap Z)$ is a zero Lebesgue measure set in $\R^k$. }).

\begin{proposition}\label{propdiffcal2}
Let $u\prec_\l L$ and $\g:[a,b]\longrightarrow M$ a $(u,L)$-calibrated curve.
Then,  for every $t\in(a,b)$, the derivative of $u$ at $\g(t)$ exists. 
\end{proposition}

\begin{proof}
The proof is similar to \cite[Theorem 4.3.8 ii)]{Fathibook}. Fix $t\in [a,b]$ and let $x:=\g(t)$. Without any loss of generality, we assume to be working in a coordinate chart $\varphi: U \longrightarrow \R^n$ on $M$ and assume that  $\g([a,b])\subset U$ (otherwise, we take a smaller interval containing $t$). To simplify notation, we identify $U$ with $\R^n$ via $\varphi$.\\
For every $y\in U$, we construct a new curve defined on $[a,t]$, such that at time $t$ it passes through $y$. Define this curve $\g_y:[a,t] \longrightarrow U$ by
$$
\g_y(s) = \g(s) + \frac{s-a}{t-a}(y-x).
$$
We have that $\g_y(a)=\g(a)$ and $\g_y(t)=y$. Observe that $\g_x$ coincides with $\g$ on the interval $[a,t]$. Since $u\prec_\l L$ we obtain:
$$
e^{\l t} u(y) - e^{\l a} u(\g_y(a)) \leq \int_a^t e^{\l s} L(\g_y(s), \dot{\g}_y(s)) ds
$$
which implies
\begin{eqnarray*}
u(y) &\leq& e^{-\l t } \left( e^{\l a} u(\g(a)) + \int_a^t e^{\l s} L(\g_y(s), \dot{\g}_y(s)) ds \right)\\
&=& e^{-\l t } \left( e^{\l a} u(\g(a)) + \int_a^t e^{\l s} L\left(
\g(s) + \frac{s-a}{t-a}(y-x), \dot{\g}(s) + \frac{y-x}{t-a}\right) ds \right)\\
&=:& \psi_+(y).
\end{eqnarray*}
Observe that $\psi_+$ is $C^1$ (actually, since $\g$ is as smooth as $L$, it is as smooth as $L$) and that we have equality at $x$.\\
\noindent Similarly, for every $y\in U$ we construct a curve defined on $[t,b]$ that at time $t$ it passes through $y$. 
Define this curve $\s_y:[a,t] \longrightarrow U$ by
$$
\s_y(s) = \g(s) + \frac{b-s}{b-t}(y-x).
$$
We have that $\s_y(b)=\g(b)$ and $\s_y(t)=y$. Observe that $\s_x$ coincides with $\g$ on the interval $[t,b]$. Since $u\prec_\l L$ we obtain:
$$
e^{\l b} u(\s_y(b)) - e^{\l t} u(y) \leq \int_t^b e^{\l s} L(\s_y(s), \dot{\s}_y(s)) ds
$$
which implies
\begin{eqnarray*}
u(y) &\geq& e^{-\l t } \left( e^{\l b} u(\g(b)) - \int_t^b e^{\l s} L(\s_y(s), \dot{\s}_y(s)) ds \right)\\
&=& e^{-\l t } \left( e^{\l b} u(\g(b)) + \int_t^b e^{\l s} L\left(
\g(s) + \frac{b-s}{b-t}(y-x), \dot{\g}(s) - \frac{y-x}{b-t}\right) ds \right)\\
&=:& \psi_-(y).
\end{eqnarray*}
Observe that also $\psi_-$ is $C^1$ and that we have equality at $x$.\\
In conclusion, we have that
\begin{equation}\label{ineqpsi}
\psi_-(y) \leq u(y) \leq \psi_+(y) \qquad \forall\; y\in M
\end{equation}
with equality at $x$. Observe that the $C^1$ function $\psi_+ -\psi_- \geq 0$ and vanishes at $x$, therefore 
$\nabla(\psi_+ -\psi_-)(x)=0$. If we denote $p:= \nabla \psi_+ (x)=  \nabla \psi_-(x)$, it is easy to check that $u$ is differentiable at
$x$ and that $\nabla u(x)=p$. 
In fact, by definition of derivative, using that 
$\psi_-(x) =  \psi_+(x)=u(x)$, we have that 
$
\psi_\pm = u(x) + p\cdot (y-x) + r_\pm (y-x),
$
where $r(h)=o(h)$ as $h\rightarrow 0$. If we use it to rewrite inequality (\ref{ineqpsi}), we obtain:
$$
u(x) + p\cdot (y-x) + r_- (y-x) \leq u(y) \leq u(x) + p\cdot (y-x) + r_+ (y-x)
$$
which implies 
$$
u(y) = u(x) + p\cdot (y-x) + o(\|y-x\|).
$$
This clearly means that $u$ is differentiable at $x$ and $\nabla u(x)=p$.
\end{proof}

\vspace{10 pt}

Using these observations (and keeping in mind the analogy with the conservative case), we  provide the following definition.\\

 A function $u: M \longrightarrow \R$ is  called a {\it weak KAM solution} to the $\l$-discounted Hamilton-Jacobi equation if:
 \begin{itemize}
 \item[(i)] $u\prec_\l L$;
 \item[(ii)] for every $x\in M$ there exists $\g: (-\infty,0] \longrightarrow M$ with $\g(0)=x$, which is $(u,L)$-$\l$-calibrated.\\
 \end{itemize}

\begin{remark}
Suppose  that $u\prec_\l L$ and that there exists a $(u,L)$-calibrated curve $\g:(-\infty,0] \longrightarrow M$ such that $\g(0)=x_0$.
Then, for all $t<0$ we have:
$$
u(x_0) - e^{\l t} u(\g(t))  = \int_t^0 e^{\l s} L(\g(s),\dot{\g}(s))\, ds.
$$
If we take the limit as $t\rightarrow -\infty$, since $u$ is bounded we obtain:
$$
u(x_0) = \int_{-\infty}^0 e^{\l s} L(\g(s),\dot{\g}(s))\, ds.
$$
In particular, since $u\prec_\l L$ (see (\ref{dominated}))
\begin{eqnarray*}
 u(x_0) = \inf_\s \left( \int_{-\infty}^0 e^{\l s} L(\s(s),\dot{\s}(s))\, ds\right),
\end{eqnarray*}
where the infimum (which in this case is a minimum) is taken over all continuous piecewise $C^1$ curves $\s: (-\infty,0] \longrightarrow M$ such that $\s(0)=x_0$.\\
Therefore, if such a weak KAM solution exists, then its value at $x_0$ is determined uniquely. We shall see in Proposition \ref{theoremonu} that there exists a function $\u$ satisfying this condition at each point $x\in M$ and, as a consequence of what we have just pointed out, it is unique.
\\
\end{remark}

\noindent Inspired by this, we define the following function. For every $x\in M$, let
\begin{equation}\label{defu}
\u(x) = \inf_\s \left( \int_{-\infty}^0 e^{\l s} L(\s(s),\dot{\s}(s))\, ds\right),
\end{equation}
where the infimum is taken over all continuous piecewise $C^1$ curves $\s: (-\infty,0] \longrightarrow M$ such that $\s(0)=x$. \\
This function has these important properties, proved in \cite[Appendix 2]{DFIZ} and references therein.

\begin{proposition}\label{theoremonu}
Let $\u$ be defined as above. Then:
\begin{enumerate}
\item $\u$ is well-defined and Lipschitz continuous. In particular, the Lipschitz constant does not depend on $\l$ but only on $L$.
\item $\u \prec_\l L$.
\item For every $x$, there exists a curve $\g_x: (-\infty,0] \longrightarrow M$ which achieves the infimum in (\ref{defu}). In particular, $\g_x$ is $(\u, L)$-calibrated, which implies that
for any $t<0$:
$$ \u(x) =   e^{\l t} \u(\g_x(t)) - \int_t^0 e^{\l s} L(\g_x(s), \dot{\g}_x(s))ds.   $$
\item There exists a constant $A$ (depending on $L$, but not on $\l$) such that for all $x\in M$, $\|\dot{\g}_x\|_\infty \leq A$. \\
\end{enumerate}
\end{proposition}

\vspace{10 pt}
Observe that since $\u$ is Lipschitz, by Rademacher's theorem it is differentiable everywhere and therefore it satisfies the equation
$$
\l \u + H(x, d\u(x))=0 \quad {\rm a.e.}
$$
In particular, for $\l>0$ this equation satisfies a strong comparison principle (see for example \cite[Th\'eor\`eme 2.4]{Barles}) and therefore this is the unique solution.\\


\section{The Aubry set}\label{sec:aubry}
In this section we are going to define the analogue of the Aubry set in the conformally symplectic framework. \\

For every $(x,v)\in TM$ let us denote by $\g_{(x,v)}$ the projection on $M$ of the corresponding orbit, \ie $\g_{(x,v)}(t)=\pi\left( \Phi^t_{L,\l}(x,v)\right)$ for all $t\in \R$. 

We define the following set. 
\begin{equation}\label{defcSt}
\cSt_{L,\l} := \left\{ (x,v)\in TM \; \mbox{s.t. the curve}\; \g_{(x,v)} \;  \mbox{is}\; (\u,L)\mbox{-calibrated on}\; (-\infty,0]
\right\}.
\end{equation}
\noindent We note that the following properties of $\cSt_{L,\l}$ hold.

\begin{itemize}
\item[s.1)] $\cSt_{L,\l}\neq \emptyset$, as it follows from item (3) in Proposition \ref{theoremonu}. More specifically, $\pi(\cSt_{L,\l}) = M$. In general this projection does not  need to be injective.

\item[s.2)] $\cSt_{L,\l}$ is bounded, as it follows from item (4) in  Proposition \ref{theoremonu}.

\item[s.3)] $\cSt_{L,\l}$ is backward-invariant, \ie $\Phi_{L,\l}^{-t}(\cSt_{L,\l}) \subseteq \cSt_{L,\l}$ for all $t\geq 0$. Essentially, this means that if $(x,v)\in \cSt_{L,\l}$, then $\Phi_{L,\l}^{-t}(x,v)\in \cSt_{L,\l}$ for all $t\geq0$, but it is straightforward from the calibration condition of $\g_{(x,v)}$ and the fact that  $\g_{ \Phi_{L,\l}^{-t}(x,v) }(s) = \g_{(x,v)}(s-t)$ for all $s\leq 0$ (see Remark \ref{remarkcalibrated}, item {\it iii})). Actually, one has that $\g_{ \Phi_{L,\l}^{-t}(x,v) }$ is calibrated on the larger interval $(-\infty,t]$.

\item[s.4)] For every $t>0$, $\Phi_{L,\l}^{-t}(\cSt_{L,\l})$ is compact.  In fact, fix $t>0$ and take any sequence $\{(x_n,v_n)\}_n\subset \Phi_{L,\l}^{-t}(\cSt_{L,\l})$.
We consider the corresponding (minimizing) curves $\g_n = \g_{(x_{n},v_{n})}$ which are calibrated on $(-\infty,t]$, as showed in item s.3.  If we apply \cite[Theorem 6.4]{DFIZ} (plus a diagonal argument), we obtain that there
exists a subsequence $\g_{n_k}$ converging to a curve $\bar{\g}:(-\infty,t] \longrightarrow M$  uniformly on compact subsets of $(-\infty,t]$. Since the action-functional is lower semi-continuous and $\u$ is $\l$-dominated, the curve $\bar{\g}$ is $(\u,L)$-calibrated on $(-\infty,t]$, hence $C^1$. Therefore $(\bar{x},\bar{v})=(\bar{\g}(0), \dot{\bar{\g}}(0)) \in \Phi_{L,\l}^{-t}(\cSt_{L,\l})$ and clearly, for every $s\leq t$, $\bar{\g} = \pi\Phi_{L,\l}^{s}(\bar{x},\bar{v})$ where $\pi:TM\rightarrow M$ denotes the canonical projection. 
Using Propositions \ref{propdiffcal1} and \ref{propdiffcal2}, the properties of the Legendre transform $\cL_L$ and the above convergence result for $\g_{n_k}$, we conclude:
\begin{eqnarray*}
(x_{n_k},v_{n_k}) &=& \left(\g_{n_k}(0),\dot{\g}_{n_k}(0)\right) = \cL_{L}^{-1}\left( \g_{n_k}(0),d\u({\g}_{n_k}(0)) \right) \\
&& \stackrel{n_k\rightarrow +\infty}{\longrightarrow} \;
\cL_{L}^{-1}\left( \bar{\g}(0),d\u(\bar{\g}(0)) \right) = (\bar{x},\bar{v}).
\end{eqnarray*}
Hence, $\Phi_{L,\l}^{-t}(\cSt_{L,\l})$ is compact for any $t>0$.\\
\end{itemize}

We are now ready to define the analog of the Aubry set as

\begin{equation} \label{defA}
\cAt_{L,\l} := \bigcap_{t\geq 0} \Phi^{-t}_{L,\l}(\cSt_{L,\l}) = \bigcap_{t>0} \Phi^{-t}_{L,\l}(\cSt_{L,\l}),
\end{equation}
where the last equality follows from the fact  $\cSt_{L,\l}$ is backward-invariant, see item s.3 above.

Let us now describe some properties of $\cAt_{L,\l}$.
\begin{itemize}
\item[a.1)] $\cAt_{L,\l}\neq \emptyset$ since it is intersection of a decreasing family of compact sets. In particular, $\cAt_{L,\l}$ is compact.

\item[a.2)] $\cAt_{L,\l}$ is invariant. 

It is a consequence of its definition (\ref{defA}), using the fact that $\cSt_{L,\l}$ is  backward-invariant. More precisely, $\cAt_{L,\l}$ is backward-invariant being the intersection of backward-invariant sets. Moreover, we have that for $s>0$, $\Phi^{s}_{L,\l}(\cAt_{L,\l}) \subseteq \cAt_{L,\l}$. In fact:
\begin{eqnarray*}
\Phi^{s}_{L,\l}(\cAt_{L,\l}) &=& \Phi^{s}_{L,\l} \left( \bigcap_{t\geq0} \Phi^{-t}_{L,\l}(\cSt_{L,\l}) \right) = \bigcap_{t\geq0} \Phi^{-t+s}_{L,\l}(\cSt_{L,\l})\\
&=&\bigcap_{t\geq-s} \Phi^{-t}_{L,\l}(\cSt_{L,\l}) \subseteq \bigcap_{t\geq0} \Phi^{-t}_{L,\l}(\cSt_{L,\l}) = \cAt_{L,\l}.
\end{eqnarray*}

In particular, every invariant set $\Lambda \subset \cSt_{L,\l}$ must be contained in $\cAt_{L,\l}$. In fact, if $\L$ is invariant and contained in $\cSt_{L,\l}$, then $\Phi_{L,\l}^{t} (\L) \subseteq \cSt_{L,\l}$ for all $t\geq 0$. Hence,   $\L \subseteq \Phi_{L,\l}^{-t} (\cSt_{L,\l})$ for all $t\geq0$. It follows from the definition of $\cAt_{L,\l}$ in (\ref{defA}) that $\L \subseteq \cAt_{L,\l}$.
\item[a.3)] Orbits starting in $\cAt_{L,\l}$ have special calibrating properties. Namely, if $(x,v) \in \cAt_{L,\l}$ then  
the curve  $\g_{(x,v)}$ is $(\u,L)$-calibrated on $(-\infty,+\infty)$.
In fact, observe that  if $(x,v) \in \cAt_{L,\l}$, then  $(x,v) \in \Phi_{L,\l}^{-t}(\cSt_{L,\l})$ for all $t\geq0$. So, as we have remarked above in s.3, the curve $\g_{(x,v)}$ is calibrated on  $(-\infty,t]$. Since this is true for all $t>0$, then this proves the assertion. \\
In particular, observe that calibration implies that they are action-minimizers (see Remark \ref{remarkcalibrated}, ({\it iv})).

\item[a.4)] The projection
$\pi : \cAt_{L,\l}  \longrightarrow M$ such that  $\pi(x,v)=x$ is injective. In fact, if $(x,v)\in \cAt_{L,\l}$, then it can be deduced from Propositions \ref{propdiffcal1} and \ref{propdiffcal2}, that $\u$ is differentiable at $\g_{(x,v)}(0)=x$ and that $(x,v)= \cL_{L}^{-1}(x, d\u(x))$; hence, $v$ is determined uniquely by $x$. More specifically, 
$$
v= \frac{\partial H}{\partial p}(x, d\u(x)) \qquad\Longleftrightarrow \qquad d\u (x) = \frac{\partial L}{\partial v}(x, v).
$$
In particular, if we denote by $\cA_{L,\l} := \pi \left( \cAt_{L,\l} \right)$, then we have that 
\begin{equation}
\cAt_{L,\l} = \left\{  \left(x, \frac{\partial H}{\partial p}(x, d\u (x)\right):\quad x\in \cA_{L,\l} \right\}.
\end{equation}
Observe that the map $du$ is well-defined on $\cA_{L,\l}$ and it coincides with $\frac{\partial L}{\partial v}  \circ  (\pi_{\big|{\cAt_{L,\l}}})^{-1}$. So it follows from the compactness of $\cAt_{L,\l}$ that this map is Lipschitz, with a Lipschitz constant which is independent of $\l$ (this latter property is a consequence of  item 4 in Proposition \ref{theoremonu}. This can be summarized by saying that
$
\pi: \cAt_{L,\l} \longrightarrow \cA_{L,\l}
$
is a bi-Lipschitz homeomorphism. This is the analogue of Mather's graph theorem  for the conservative case (see \cite[Theorem 2]{Mather91}).\\
\end{itemize}

\noindent Let us briefly describe the relation between this set and invariant exact Lagrangian graphs.

\begin{proposition}\label{AubryLaggraph}
Let $\L$ be a $C^1$ invariant exact Lagrangian graph for $\Phi_{H,\l}$. Then
$$
\Lambda = \cL_{L}(\cAt_{L,\l}).\\
$$
\end{proposition}

\begin{proof}
Since $\Lambda$ is an exact Lagrangian graph, then $\Lambda={\rm Graph}(dv)$ for some $v\in C^2(M)$. It follows from the invariance of $\Lambda$, that $v$ is a classical solution to
the $\lambda$-discounted Hamilton-Jacobi equation (see Proposition \ref{invariantlagrangiangraph}). As we have already remarked before,
for $\l>0$ this equation satisfies a strong comparison principle (see for example \cite[Th\'eor\`eme 2.4]{Barles}) and therefore it admits a unique solution (including weak solutions), which implies that $v=\u$.\\
For simplifying the notation, in the following we denote $\widetilde{\L} = \cL_L^{-1}(\Lambda)$. It follows from Proposition \ref{existencesolutioncase} that $\widetilde{\L}\subseteq \cSt$ and since it is invariant
$$ \widetilde{\L} = \Phi^t_{L,\l}(\widetilde{\L}) \subseteq \Phi^t_{L,\l}(\cSt) \qquad \qquad \forall\, t\in \R.$$
In particular,  we can conclude from (\ref{defA}) that 
$$
\cAt_{L,\l} = \bigcap_{t\geq0} \Phi^{-t}_{L,\l}(\cSt_{L,\l}) \supseteq \widetilde\Lambda,
$$
and because of the graph property (see item (a.4) after (\ref{defA})), they must coincide:  $\cAt_{L,\l} = \widetilde\Lambda$. This concludes the proof.
\end{proof}

\begin{remark}  
In \cite{CCD, CCD2} the authors studied the persistence of KAM tori ({\it i.e.},  smooth invariant Lagrangian graphs on which the dynamics is conjugated to a rotation) under small perturbations of conformally symplectic vector fields. 
Observe that whenever a KAM torus exists, then it is unique and it coincides with the Aubry set defined above (this follows from Proposition \ref{AubryLaggraph} and Remark \ref{remnonexact}).
\end{remark}

\vspace{10 pt}


\section{Global Behaviour and Attractiveness}\label{sec:attract}
In this section we want to discuss global properties of the  flow and prove the existence (and the properties) of an attracting region for the orbits, which contains the Aubry set $\cAt_{L,\l}$ as the unique invariant set  in its ``frontier''. \\

We consider the following function
$${F_\l}(x,p)= \l \u(x) + H(x,p)$$
and the following disjoint subsets of $T^*M$:
\begin{eqnarray*}
\cZ^0_{F_\l} &:=& \{(x,p)\in T^*M: \; {F_\l}(x,p)=0\}\\
\cZ^+_{F_\l} &:=& \{(x,p)\in T^*M: \; {F_\l}(x,p)>0\}\\
\cZ^-_{F_\l} &:=& \{(x,p)\in T^*M: \; {F_\l}(x,p)<0\}.\\
\end{eqnarray*}
\begin{remark}
These three sets form a partition. Moreover, $\cZ^0_{F_\l}$ is compact and $\cZ^\pm_{F_\l}$ are open. 
It follows from the superlinearity of $H$ that  $\cZ^+_{F_\l}$ is unbounded, while $\cZ^-_{F_\l}$ is bounded.
\end{remark}
\noindent We are going to use the  these sets  to study the global dynamics of the system. \\
\noindent To do so, let us investigate the variation of ${F_\l}$ in the direction of the flow.
Recall that $\u$ is only locally Lipschitz continuous, hence, it is differentiable almost everywhere. Let us denote this measure zero set of non-differentiability by
\begin{eqnarray*}
\cN &:=&  \{x \in M: \u\;\mbox{is not differentiable at}\; x\}.
\end{eqnarray*}
Observe that the problem of being non-differentiable for ${F_\l}$ comes only from the $\u$ component; hence, ${F_\l}$ is differentiable at $(x,p)$ if and only if $x\not \in \cN$ (which is also a measure zero set in $T^*M$).\\
Let us start by observing how the Hamiltonian varies along the orbits. Using the equation of motion (\ref{eqmotion}) and the Legendre-Fenchel (in)equality (\ref{LegFenineq}), we obtain:
\begin{eqnarray} \label{derivHamiltonian}
\frac{d}{dt}H(x(t),p(t)) &=& \frac{\partial H}{\partial x} (x(t),p(t)) \cdot \dot{x}(t) + 
\frac{\partial H}{\partial p} (x(t),p(t)) \cdot \dot{p}(t) \nonumber\\
&=& - \l \, \big \langle  p(t), \frac{\partial H}{\partial p} (x(t),p(t))\big  \rangle \\
&=& - \l \, \big[L( \cL_L^{-1}(x(t), {p}(t))) + H(x(t), p(t)) \big] \nonumber \\
&=& - \l \, \big[L(x(t), \dot{x}(t))  + H( \cL_L(x(t), \dot{x}(t))) \big]. \nonumber
\end{eqnarray}
We use it to prove
\begin{lemma} \label{derivF}
For every $(x,p) \in T^*M$ such that $x\not \in \cN$,    $ \langle d{F_\l}(x,p), X_H(x,p) \rangle \leq -\l {F_\l}(x,p)$.
\end{lemma}
\begin{proof}
Let $x\not \in \cN$, $p\in T^*_xM$ and
denote  by
$(x,v)=\cL_{L}^{-1}(x,p)$. Using (\ref{derivHamiltonian}) and the Legendre-Fenchel inequality, we have:
\begin{eqnarray*}
\langle d{F_\l}(x,p), X_H(x,p) \rangle &=& \l \langle d\u(x), v \rangle +  \frac{d}{dt}H(\Phi_{H,\l}^t(x,p))_{\big|t=0} \\
&=& \l \langle d\u(x), v \rangle  - \l [L(x,v) + H(x,p)]\\
&\leq& \l \left[
L(x,v) + H(x,d\u(x)) - L(x,v) - H(x,p)
\right]\\
&=&
\l \left[ H(x,d\u(x)) - H(x,p)\right]\\
&=&
\l \left[ -\l \u(x) - H(x,p)\right]\\
&=& - \l {F_\l}(x,p),
\end{eqnarray*}
where we used that $\l \u(x) + H(x,d\u (x))=0$ at points of differentiability of $\u$.
\end{proof}
\noindent We can now start by studying the set $\cZ^0_{F_\l}$.

\begin{lemma}\label{lemma2}
We have that $ \cL_{L}(\cSt_{L,\l}) \subseteq \cZ^0_{F_\l}$; in particular,  $ \cL_{L}(\cAt_{L,\l}) \subseteq \cZ^0_{F_\l}$.  
\end{lemma}

\begin{proof}
It is enough to prove the first statement, being the second a clear consequence.
Let $(x,p)\in  \cL_{L}(\cSt_{L,\l})$ and let $(x,v)=\cL_{L}^{-1}(x,p)$. We denote their respective orbits by $(x(t), p(t))=\Phi_{H,\l}^t(x,p)$  and $(x(t), v(t))=\Phi_{L,\l}^t(x,v)$ with $t\in \R$. Using the definition of $\cSt_{L,\l}$ in (\ref{defcSt}), property (\ref{derivHamiltonian}) and the fact that $\cSt_{L,\l}$ is bounded and backward-invariant (hence, the Hamiltonian is bounded along the backward orbit), we obtain:
\begin{eqnarray*}
\l \u(x) &=&  \l \int_{-\infty}^0 e^{\l t} L(x(t),v(t))\, dt \\
 &=&  \int_{-\infty}^0 e^{\l t} \left[  \l \, \big(L(x(t),v(t)) +  H(x(t),p(t)) \big)-\l H(x(t),(p(t)) \right]\,dt\\
&=& - \int_{-\infty}^0 e^{\l t} \left[
\frac{d}{dt}\Big(H(x(t),p(t) \Big)
+
\l H(x(t),p(t))
\right]\, dt\\
&=& - \int_{-\infty}^0 \frac{d}{dt}\left(e^{\l t}H(x(t),p(t) \right)\, dt\\
&=&- H(x,p). 
\end{eqnarray*}
Therefore, ${F_\l}(x,p)=0$.
\end{proof}

\vspace{10 pt}

\noindent A sort of converse of the previous lemma holds.\\

\begin{lemma}\label{propZzero}
Let $(x,p) \in \cZ^0_{F_\l}$. If 
\begin{equation}\label{ipotesi}
\lim_{t\rightarrow -\infty} e^{\l t} H(\Phi^t_{H,\l}(x,p)) = 0,
\end{equation}
 then
$(x,p)\in \cL_{L}(\cSt_{L,\l})$. \\
In particular, all invariant sets in $\cZ^0$ are contained in $\cL_{L}(\cAt_{L,\l})$.\\
\end{lemma}

\begin{proof}
Let $(x,p)\in \cZ^0_{F_\l}$, \ie $\l \u(x) + H(x,p) = 0$.
Let us denote $(x,v)=\cL_{L}^{-1}(x,p)$ and the respective orbits by $(x(t), p(t))=\Phi_{H,\l}^t(x,p)$  and $(x(t), v(t))=\Phi_{L,\l}^t(x,v)$ with $t\in \R$.\\
 We want to prove that $(x,v)\in \cSt_{L,\l}$. Using  hypothesis (\ref{ipotesi})
   and property (\ref{derivHamiltonian}), we can deduce the following estimate:
\begin{eqnarray*}
\l \u(x) &=& - H(x,p) = - \int_{-\infty}^0 \frac{d}{dt}\left(e^{\l t}H(x(t),p(t) \right)\, dt\\
&=& - \int_{-\infty}^0 e^{\l t} \left[\l H(x(t),p(t)) + 
\frac{d}{dt}\Big(H(x(t),p(t) \Big)
\right]\, dt\\
 &=& - \int_{-\infty}^0 e^{\l t} \left[\l H(x(t),p(t)  -\l  L(x(t),v(t)) -\l H(x(t),(p(t)) 
\right]\, dt \\
 &=& \l \int_{-\infty}^0 e^{\l t} L(x(t),v(t))\, dt.
\end{eqnarray*}
Hence, the orbit $(x(t), v(t))$ for $t\in (-\infty, 0]$ achieves the minimum in the definition of $\u$ and it is therefore
$(\u,L)$-calibrated on $ (-\infty, 0]$. It follows from the definition of $\cSt_{L,\l}$ in (\ref{defcSt})  that $(x,v)\in \cSt_{L,\l}$.\\ 
To prove the last part, let us assume that $\Lambda \subseteq \cZ^0_{F_\l}$ is an invariant set. Observe that  $\Lambda$ being bounded and invariant, then 
for each $(x,p)\in \Lambda$ we have that $H(\Phi^t_{H,\l}(x,p))$ is bounded for all $t$. In particular, it follows from the previous part that $(x,p)  \in \cL_{L}(\cSt_{L,\l})$ and therefore
$\Lambda \subseteq  \cL_{L}(\cSt_{L,\l})$. Since all invariant sets in $\cSt_{L,\l}$ are contained in $\cAt_{L,\l}$ (see item a.2 after (\ref{defA})), then we can conclude 
that $\Lambda \subseteq  \cL_{L}(\cAt_{L,\l})$.
\end{proof}

\noindent The function ${F_\l}$ is a sort of {\it Lyapunov function} for the system and it allows to deduce useful information on 
the global properties of the flow.  
Let us first recall some definitions. We denote by $\Omega_{\infty}(x,p)$ the {\it $\omega$-limit set} of $(x,p)$ defined as the set of points $(\bar x,\bar p)\in T^*M$ for which there exists a sequence $(t_k)$, $t_k\to +\infty$ as $k\to +\infty$ such that 
  \[
  \lim_{k\to +\infty}\Phi_{H,\l}^{t_k}(x,p)=(\bar x,\bar p)
  \]
  Similarly, if $E\subseteq T^*M$ we denote by $\Omega_{\infty}(E)$ the set of future accumulation points of orbits starting in $E$.\\

\begin{proposition} \label{propLyapunov}
For  every $(x,p)\in T^*M$ and every $t>0$
\begin{equation}\label{lyap}
{F_\l}(\Phi^t_{H,\l}(x,p)) \leq {F_\l}(x,p) e^{-\l t}.
\end{equation}
As a consequence, the set  $\cZ^0_{F_\l}\cup \cZ^-_{F_\l}$ is an attracting set. 
 In particular, it is forward invariant and
$$\Omega_{\infty}(T^*M)\subseteq \cZ^0_{F_\l}\cup \cZ^-_{F_\l},$$
i.e., the $\omega$-limit points of any orbit are contained in $\cZ^0_{F_\l}\cup \cZ^-_{F_\l}$.
\end{proposition}

\begin{proof}
It is sufficient to prove (\ref{lyap}). The forward-invariance of $\cZ^0_{F_\l}\cup \cZ^-_{F_\l}$, in fact, follows immediately from  (\ref{lyap}); moreover, 
since ${F_\l}$ is continuous, {(\ref{lyap}) implies that} $\Omega_{\infty}(T^*M) \subseteq {F_\l}^{-1}((-\infty,0])=\cZ^0_{F_\l}\cup \cZ^-_{F_\l}$.\\
If ${F_\l}$ was differentiable everywhere, then in order to prove (\ref{lyap}) it would be sufficient to use Lemma \ref{derivF}; however, ${F_\l}$ is a-priori only locally Lipschitz, so that inequality  holds almost  everywhere. This issue can be solved by means of a standard argument (for example, see  also \cite[Lemma 1.5 ({\it i})]{ABC}).\\
Suppose that (\ref{lyap}) does not hold: this means that there exist $(x_0,p_0)\in T^*M$ and $t>0$ such that
$$
{F_\l}(\Phi^t_{H,\l}(x_0,p_0)) - {F_\l}(x_0,p_0) e^{-\l t} =: \delta > 0.
$$ 
Since both ${F_\l}$ and $\Phi^t_{H,\l}$ are locally Lipschitz, then we can find $r > 0$ and $C>0$ such that
\begin{equation*}
{F_\l}(\Phi^t_{H,\l}(x,p)) - {F_\l}(x,p) e^{-\l t} \geq  \delta - C \,d((x,p), (x_0,p_0)) \qquad \forall\; (x,p)\in B_r(x_0,p_0),
\end{equation*}
where $d$ denotes the  distance function on $T^*M$ induced by the Riemannian metric and $B_r(x_0,p_0)$ is the corresponding ball of radius $r$ centered at $(x_0,p_0)$. 
By integrating this inequality on a ball of radius $\rho\leq r$ we obtain ($\vol$ denotes the Riemannian volume on $T^*M$):
\begin{eqnarray*}
&& \int_{B_\r(x_0,p_0)} \left[
{F_\l}(\Phi^t_{H,\l}(x,p)) - {F_\l}(x,p) e^{-\l t} 
\right] \,d\,\vol (x,p) \\
&& \quad 
\geq \ 
\delta \, \vol (B_\r(x_0,p_0)) - 
\int_{B_\r(x_0,p_0)} 
\!\!\!\!\!\!\!\! C \,d((x,p), (x_0,p_0)) \;d\,\vol (x,p) \\
&& \quad 
\geq \ 
(\delta - C\,\r) \, \vol (B_\r(x_0,p_0)). 
\end{eqnarray*}
Therefore, if $0<\rho < \frac{\delta}{C}$ we have
\begin{equation} \label{integralispositive}
\int_{B_\r(x_0,p_0)} \left[
{F_\l}(\Phi^t_{H,\l}(x,p)) - {F_\l}(x,p) e^{-\l t} 
\right] \,d\,\vol (x,p) \;>\;0.
\end{equation}

On the other hand, using Tonelli's Theorem, the fact that the function $s\longmapsto {F_\l}(\Phi^s_{H,\l}(x,p))$ is locally Lipschitz continuous (hence, differentiable almost everywhere), and  the chain rule for Lipschitz continuous maps, we obtain:

{\footnotesize
\begin{eqnarray*}
&& \int_{B_\r(x_0,p_0)} \left[
{F_\l}(\Phi^t_{H,\l}(x,p)) - {F_\l}(x,p) e^{-\l t} 
\right] \,d\,\vol (x,p) \\
&& \quad 
= \
e^{-\l t }\, \int_{B_\r(x_0,p_0)} \left[
e^{\l t} {F_\l}(\Phi^t_{H,\l}(x,p)) - {F_\l}(x,p) 
\right] \,d\,\vol (x,p) \\
&& \quad 
= \ 
e^{-\l t }\,
\int_{B_\r(x_0,p_0)}   \left[
\int_0^t \frac{d}{ds} \big( e^{\l s} {F_\l}(\Phi^s_{H,\l}(x,p))\big)\,ds
\right] \,d\,\vol (x,p) \\
&& \quad 
= \
e^{-\l t }\, 
\int_0^t   \left[
\int_{B_\r(x_0,p_0)} \frac{d}{ds} \Big( e^{\l s} {F_\l}(\Phi^s_{H,\l}(x,p))\Big)\,d\,\vol (x,p)
\right] \,ds \\
&& \quad 
= \ 
e^{-\l t }\,
\int_0^t   e^{\l s}\, \left[
\int_{B_\r(x_0,p_0)} \!\!\!\!\!
 \Big( \l  {F_\l}(\Phi^s_{H,\l}(x,p)) + \frac{d}{ds}\big(
{F_\l}(\Phi^s_{H,\l}(x,p))
\big)
 \Big)\,d\,\vol (x,p)
\right] \,ds \\
&&\quad
= \ 
e^{-\l t }\,
\int_0^t   e^{\l s}\, \left[
\int_{B_\r(x_0,p_0)} \!\!\!\!\!
 \Big( \l  {F_\l}(\Phi^s_{H,\l}(x,p)) +
 \langle d{F_\l}, X_H \rangle_{\big| (\Phi^s_{H,\l}(x,p))}
 \Big)\,d\,\vol (x,p)
\right] \,ds \\
&& \quad 
\leq 0 
\end{eqnarray*}
}
where the last step follows from  the fact that, in the light of Lemma \ref{derivF}, the integrand is non positive  almost everywhere.
This conclusion is in contradiction with (\ref{integralispositive}). 
\end{proof}

\vspace{10 pt}

\begin{corollary}\label{cor1}
{\rm (1)} If  $\Omega_{\infty}(x,p) \subseteq \cZ_{F_\l}^0$, then $\Omega_{\infty}(x,p) \subseteq \cL_L(\cAt_{L,\l})$.\\
{\rm (2)} If there exists $t_0$ such that $\Phi_{H,\l}^t(x,p) \in \cZ^+_{F_\l} \cup \cZ^0_{F_\l}$ for all $t\geq t_0$, then $\Omega_{\infty}(x,p) \subseteq \cL_L(\cAt_{L,\l})$.\\
{\rm (3)} If there exists a sequence $\{t_n\}_{n\geq 0}$ such that $t_n\rightarrow +\infty$ and $\Phi_{H,\l}^{t_n}(x,p) \in \cZ^+_{F_\l} \cup \cZ^0_{F_\l}$ for all $n\geq 0$, then $\Omega_{\infty}(x,p) \subseteq \cL_L(\cAt_{L,\l})$.\\
{\rm (4)} If ${F_\l}(x,p)\geq 0$ for all $(x,p)\in T^*M$, then $\Omega_{\infty}(T^*M) \subseteq \cL_L(\cAt_{L,\l})$. In particular,  $\cAt_{L,\l}$ is a global attractor.\\
\end{corollary}

\begin{proof}
(1)  If $\Omega_{\infty}(x,p) \subseteq \cZ^0_{F_\l}$, then using  Lemma \ref{propZzero} and the fact that $\Omega_{\infty}(T^*M)$ is invariant, we can deduce that 
 $\Omega_{\infty}(x,p) \subseteq \cL_L(\cSt_{L,\l})$. In particular, we have proved (see item a.2 after (\ref{defA})) that all invariant sets in $\cL_L(\cSt_{L,\l})$ must be contained in $\cL_L(\cAt_{L,\l})$ and this completes the proof.\\
\noindent (2) If follows from the fact that $\Phi_{H,\l}^t(x,p) \in \cZ^+_{F_\l} \cup \cZ^0_{F_\l}$  for all $t\geq t_0$,  from  Proposition \ref{propLyapunov} and from the continuity of ${F_\l}$, that 
$$
0\leq {F_\l}(\Phi_{H,\l}^t(x,p)) \leq {F_\l}(\Phi_{H,\l}^{t_0}(x,p))\, e^{-\l (t-t_0)} \qquad \forall\; t\geq t_0,
$$
and therefore  $\Omega_{\infty}(x,p) \subseteq \cZ^0_{F_\l}$. The conclusion follows from part (1).\\
\noindent (3) Proceeding as in (2), we obtain
$$
0\leq {F_\l}(\Phi_{H,\l}^{t_n}(x,p)) \leq {F_\l}(\Phi_{H,\l}^{t_n}(x,p))\, e^{-\l (t_n-t_0)} \stackrel{n\rightarrow +\infty}{\longrightarrow} 0
$$
and therefore  $\Omega_{\infty}(x,p) \subseteq \cZ^0_{F_\l}$. The conclusion follows again from part (1).\\
\noindent (4) It follows easily from (2).
\end{proof}

\vspace{10 pt}

\noindent We are now ready to define the set:
\begin{equation}\label{defK}
\cK_{H,\l} := \bigcap_{t\geq 0} \Phi_{H,\l}^t (\cZ^0_{F_\l}\cup \cZ^-_{F_\l}).
\end{equation}
and prove the following proposition.\\

\begin{proposition}\label{Kattractor}
The set $\cK_{H,\l}$ is the maximal global attractor for $X_{H,\l}$ and
$$
\cL_{L}(\cAt_{L,\l})\subseteq\cK_{H,\l}.  
$$
In particular, $
\cL_{L}(\cAt_{L,\l})= \cK_{H,\l} \cap {\cZ^0_{F_\l}}.  
$
\end{proposition}

\begin{proof}
First of all, it follows from the definition that $\cK_{H,\l}$ is compact. Moreover, using an argument similar to the one in Section \ref{sec:aubry}, item a.2, we can conclude that it is invariant; in fact if $s<0$ then  
\begin{eqnarray*}
\Phi_{H, \l}^s(\cK_{H,\l}) = \bigcap_{t\geq s} \Phi_{H,\l}^t (\cZ^0_{F_\l}\cup \cZ^-_{F_\l}) \subseteq \bigcap_{t\geq 0} \Phi_{H,\l}^t (\cZ^0_{F_\l}\cup \cZ^-_{F_\l}) = \cK_{H,\l},
\end{eqnarray*}

while if $s>0$ (since  $\cZ^0_{F_\l}\cup \cZ^-_{F_\l}$ is forward-invariant, see Proposition \ref{propLyapunov}):
\begin{eqnarray*}
\Phi_{H, \l}^s(\cK_{H,\l}) &=& \bigcap_{t\geq 0} \Phi_{H,\l}^{t+s} (\cZ^0_{F_\l}\cup \cZ^-_{F_\l}) \\
&=& \bigcap_{t\geq 0} \Phi_{H,\l}^{t}( \Phi_{H,\l}^{s}(\cZ^0_{F_\l}\cup \cZ^-_{F_\l})) \\
&\subseteq& \bigcap_{t\geq 0} \Phi_{H,\l}^t (\cZ^0_{F_\l}\cup \cZ^-_{F_\l}) \subseteq \cK_{H,\l}.
\end{eqnarray*}

Moreover, it contains the Aubry set as a consequence of Lemma \ref{lemma2}, so it is not empty. 
We prove that $\cL_{L}(\cAt_{L,\l})= \cK_{H,\l} \cap {\cZ^0_{F_\l}}$. In fact, clearly $\cL_{L}(\cAt_{L,\l}) \subseteq {\cZ^0_{F_\l}}$ (see Lemma \ref{lemma2}).
On the other hand, if  $(x,p) \in \cK_{H,\l} \cap {\cZ^0_{F_\l}}$, then it follows from Lemma \ref{propZzero} and the invariance of $\cK_{H,\l}$ that $(x,p) \in \cL_{L}(\cAt_{L,\l})$.\\
In order to prove that $\cK_{H,\l}$ is a global attractor, we need to prove that it is a global attracting set. 
Recall that Proposition \ref{propLyapunov} implies that
$$
\Omega_{\infty}(T^*M) \subseteq \cZ^0_{F_\l}\cup \cZ^-_{F_\l}.
$$
Moreover, it follows  from (\ref{defK}) and the fact that $ \Omega_{\infty}(T^*M)$ is invariant ({\it i.e.}, $\Phi_{H,\l}^t(\Omega_{\infty}(T^*M)) = \Omega_{\infty}(T^*M)$ for every $t$), that
\begin{eqnarray*}
\Omega_{\infty}(T^*M) &=&  \bigcap_{t\geq 0} \Phi_{H,\l}^t(\Omega_{\infty}(T^*M))
\subseteq  \bigcap_{t\geq 0} \Phi_{H,\l}^t(\cZ^0_{F_\l}\cup \cZ^-_{F_\l}) = \cK_{H,\l}.
\end{eqnarray*}
Therefore, using the definition of attracting set, it is easy to conclude that $\cK_{H,\l}$ is an attracting set and hence, being invariant, a global attractor.
Maximality  follows from the facts that all compact invariant sets for $\Phi_{H,\l}$ must be contained in $ \cZ^0_{F_\l}\cup \cZ^-_{F_\l}$, and that
because of its definition \eqref{defK}, $\cK_{H,\l}$ is the maximal invariant set in $\cZ^0_{F_\l}\cup \cZ^-_{F_\l}$.
\end{proof}

\vspace{10 pt}


\section{Action-minimizing measures and Mather set}\label{sec:mather}
In order to define an analogue of the Mather set in the conformally symplectic case, we need first to generalize  the notion of {\it Mather measure} or {\it action-minimizing measure} (we refer to \cite{Mane,Mather91} for the conservative case).
Let us denote by $\fM_{L,\l}$ the set of Borel probability measures on $TM$ that are invariant under $\Phi_{L,\l}$ (\ie ${(\Phi^t_{L,\l}})_* \m=\m$ for every $t\in \R$) and such that
\begin{equation}\label{boundedmomentum}
\int_{TM} \|v\| d\mu <+\infty.
\end{equation}

\noindent  Hereafter, we shall consider this set endowed with the topology given by $\lim_{n\rightarrow +\infty} \m_n = \mu$ if and only if
$$\lim_{n\rightarrow +\infty} \int_{TM} f(x,v) d\mu_n= \int_{TM} f(x,v) d\mu $$
for all $f \in C_{\ell}(TM)$, \ie functions $f:T M \longrightarrow \R$ having at most linear growth:
$$
\sup_{(x,v)\in T M} \frac{|f(x,v)|}{1+\|v\|} <+\infty\,.
$$
$\fM_{L,\l}$ can be seen as a subset of the dual space $(C_{\ell})^*$. This topology is also called {\it vague topology} and it is well-known that it is metrizable. \\

\begin{remark}\label{existencemeasures}
The set $\fM_{L,\l}=\fM_{L,\l}(L)$ is non-empty. In fact, since the set $\cAt_{L,\l}$ is compact and invariant under $\Phi_{L,\l}$, then it follows from Krylov--Bogolyubov's theorem (see, for example, \cite[Sec. 2]{Mather91}) that there exists at least an invariant (Borel) probability  measure $\mu$, which clearly satisfies condition (\ref{boundedmomentum}) since it is supported on a compact set. Alternatively, one can construct invariant probability measures in the following way. For every $x\in M$, consider the minimizing $(\u,L)$-calibrated orbit $\g_x: (-\infty,0]\longrightarrow M$ such that $\g_x(0)=x$. If one considers the probability measure $\mu_T$ evenly distributed on ${\g_x}_{|[0,T]}$, then every limit point of the family $\{\mu_T\}_{T>0}$, as $T$ goes to $+\infty$, is an invariant probability measure for $\Phi_{L,\l}$ and it follows from item (4) in Proposition \ref{theoremonu} that condition (\ref{boundedmomentum}) holds; it turns out that it is supported on $\cAt_{L,\l}$.\\
\end{remark}

\noindent We can prove the following property  of invariant probability measures.  In order to simplify notation, we shall denote by $L+H: TM \longrightarrow \R$ the function 
$(L+H)(x,v)= L(x,v)  + H( \cL_L(x,v))$.

\begin{proposition}\label{integralLplusH}
Let $\m\in \fM_{L,\l}$; then, 
$$
\int_{TM} (L+H)(x,v) \, d\mu(x,v) = 0.\\
$$
\end{proposition}

\begin{proof}
Let us start noting that $\supp \mu$ is compact, since it is contained in $\Omega_\infty(T M) \subseteq \cL_L^{-1}(\cK_{H,\l})$. To prove the result is sufficient to consider  the case in which $\mu$ is ergodic. Then, using the ergodic theorem and \eqref{derivHamiltonian}, for a generic point $(x,v)\in \supp \mu$: 
\begin{eqnarray*}
\int_{TM} (L+H)(x,v) \, d\mu(x,v) &=&  \lim_{T\rightarrow +\infty}  \frac{1}{T} \int_0^T (L+H)(x(t),\dot{x}(t)) \, dt \\
&=& -\l^{-1}\lim_{T\rightarrow +\infty}  \frac{1}{T} \int_0^T \frac{d}{dt} H(\cL_L(x(t),\dot{x}(t))) \, dt \\
&=& -\l^{-1} \lim_{T\rightarrow +\infty}  \frac{H(\cL_L(x(T), p(T))) - H(\cL_L(x(0),p(0)))}{T}  \\
&=& 0
\end{eqnarray*}
where, in the last equality, we used that $H \circ \cL_L$, being continuous, is bounded on $\supp \m$.
\end{proof}

\vspace{10 pt}

\begin{remark}
In particular, if $\m\in \fM_{L,\l}$,  then $\int_{TM} L \, d\mu  = - \int_{TM} H\circ \cL_L \, d\mu$. Hence, the averaged action coincides with the averaged energy, as it happens in the conservative case: in that case the energy is constant along the orbit and its value coincides with the {\it minimal averaged action} (also called {\it Mather's $\alpha$ function} or {\it  Ma\~n\'e critical value}; see, for example, \cite{Mather91, Fathibook, Sorrentinobook}.\\
\end{remark}

From Remark \ref{existencemeasures} we have that there exist some $\mu \in \fM_{L,\l}$ that are supported in $\cAt_{L,\l}$. We would like to characterize all of them. Let us start with the following observation. Consider the function $\u: M\longrightarrow \R$ defined in (\ref{defu}) and let $\nu \in \fM_{L,\l}$. Since $\nu$ is an invariant measure, then ${(\Phi^t_{L,\l})}_* \nu = {(\Phi^t_{L,\l})}^* \nu=\nu$ for all $t\in \R$. Moreover, using the definition of $u_\l$ and Fubini Theorem,  we obtain:
\begin{eqnarray} \label{computeminimalaction}
\int_{TM}  \l \, \u(x) \,d\nu(x,v) &\leq& \l \int_{TM} \left( 
\int_{-\infty}^0 e^{\l s} L( \Phi^s_{L,\l}(x,v))\, ds
  \right)\,d\nu(x,v) \nonumber \\
  &=&
 \l \int_{-\infty}^0 e^{\l s}   \left( \int_{TM} L( \Phi^s_{L,\l}(x,v))\, d\nu(x,v)
  \right)\,ds  \nonumber \\
 &=&
 \l \int_{-\infty}^0 e^{\l s}   \left( \int_{TM} L( x,v)\, d (\Phi^s_{L,\l})^* \nu(x,v)
  \right)\,ds \nonumber \\
 &=&
 \l \int_{-\infty}^0 e^{\l s}   \left( \int_{TM} L( x,v)\, d \nu(x,v)
  \right)\,ds \nonumber \\
   &=&
 \l \left( \int_{TM} L( x,v)\, d \nu(x,v) \right) \cdot \left(\int_{-\infty}^0 e^{\l s}   ds \right)\nonumber \\
&=& \int_{TM} L( x,v)\, d \nu(x,v).
\end{eqnarray}

\vspace{10 pt}

\noindent The following characterization holds.

\begin{proposition}\label{propminimmeasure}
Let $\mu \in \fM_{L,\l}$. Then:
$$
\int_{TM} (L- \l \u)\, d\mu \geq 0.
$$
Moreover, 
$$
\int_{TM} (L-\l \u) \,d\m =0 \qquad \Longleftrightarrow \qquad \supp \mu \subseteq \cAt_{L,\l}.\\
$$
\end{proposition}

\begin{proof}
The fact that $\int_{TM} (L- \l \u)\, d\mu \geq 0$ follows from (\ref{computeminimalaction}). Hence, let us prove the second part.\\
If $\supp \m \subseteq \cAt_{L,\l}$, then for every $(x,v) \in \supp \mu$ we have that 
$\Phi^s_{L,\l}(x,v)= (\g_{x}(s), \dot{\g}_x(s))$ for all $s\in (-\infty, 0]$, where $\g_x$ is the curve achieving the infimum in the definition of $\u(x)$ (see item (3) in Proposition \ref{theoremonu}). Therefore, proceeding as in (\ref{computeminimalaction}) we get:
\begin{eqnarray*}
\int_{TM}  \l \, \u(x) \,d\mu(x,v) &=& \l \int_{TM} \left( 
\int_{-\infty}^0 e^{\l s} L( \Phi^s_{L,\l}(x,v))\, ds
  \right)\,d\mu(x,v) \nonumber \\
  &=& \ldots \ =\  \int_{TM} L( x,v)\, d \mu(x,v).
\end{eqnarray*}
Hence, $\int_{TM} (L-\l \u) \,d\m =0$.\\
On the other side, if $\int_{TM} (L-\l \u) \,d\m =0$, then it follows from (\ref{computeminimalaction}) that for $\mu$-almost every $(x,v) \in \supp \mu$ we have that 
$$
\u(x)= \int_{-\infty}^0 e^{\l s} L( \Phi^s_{L,\l}(x,v))\, ds.
$$
Hence, it follows that the orbit $\Phi^s_{L,\l}(x,v)$ is $(\u,L)$-calibrated on $(-\infty,0]$ (see item (3) in Proposition \ref{theoremonu}) and therefore $(x,v) \in \cSt_{L,\l}$. In particular, using the closedness of $\cSt_{L,\l}$, we can conclude that $\supp \mu \subseteq \cSt_{L,\l}$. Since $\mu$ is invariant, then for every $t\in \R$
$$\Phi_{L,\l}^{t}(\supp \mu) = \supp \mu \subset\Phi^{t}_{L,\l}(\cSt_{L,\l}).$$
Hence,  it follows from the definition of $\cAt_{L,\l}$ in (\ref{defA}) that $\supp \mu \subseteq \cAt_{L,\l}$.
\end{proof}

\vspace{10 pt}

\noindent This result justifies the following definition:\\

We say that a measure $\mu \in \fM_{L,\l}$ is a {\it minimizing measure} if $$\int_{TM} (L- \l \u)\, d\mu = 0.\\$$

\vspace{10 pt}

\begin{remark} (i) When $\l=0$, this definition coincides with the classical definition of Mather's measures (see \cite{Mather91, Sorrentinobook}).\\
(ii) If $\mu$ is a minimizing measure, then, using Proposition \ref{integralLplusH} and the fact that $\u$ is differentiable on $\cA_{L,\l}$, we obtain:
\begin{eqnarray*}
\int_{TM} (L- \l \u)\, d\mu = 0 = \int_{TM} (L+H) \, d\mu.
\end{eqnarray*}
Hence,
\begin{eqnarray*}
\int_{TM} ( \l \u +  H \circ \cL_L)\, d\mu = 0 
\end{eqnarray*}
or equivalently
\begin{eqnarray*}
\int_{TM} ( \l \u(x) +  H (x, d\u (x))\, d \pi_{*}\mu (x) = 0, 
\end{eqnarray*}
where $\pi: TM \longrightarrow M$ denotes the projection.\\
\end{remark}

\vspace{10 pt}

Let us define the following invariant set, which , in analogy with the conservative case,  will be called the {\it Mather set}:
\begin{equation}\label{defMather}
\cMt_{L,\l} := \overline{\bigcup\left\{
\supp \m: \; \m\; \mbox{is minimizing}\right\}  }.\\
\end{equation}

\vspace{10 pt}

\noindent In analogy with what done for the Aubry set in (\ref{defA}), we describe some properties of the Mather set:\\
\begin{itemize}
\item[m.1)] $\cMt_{L,\l}\neq \emptyset$, as it follows from Remark \ref{existencemeasures} and Proposition \ref{propminimmeasure}. Moreover, it follows from the definition that $
\cMt_{L,\l}\subseteq \cAt_{L,\l}$ (see also item (a.2) after the definition of $\cAt_{L,\l}$ in (\ref{defA})).

\item[m.2)] $\cMt_{L,\l}$ is clearly invariant, since it is the closure of the union of invariant objects.

\item[m.3)]  (Graph property) Since $\cMt_{L,\l} \subseteq \cAt_{L,\l}$, then the projection $\pi : \cMt_{L,\l}  \longrightarrow M$ such that  $\pi(x,v)=x$ is injective (see item (a.4) after the definition of $\cAt_{L,\l}$ in (\ref{defA})).
In particular, 
$\pi: \cMt_{L,\l} \longrightarrow \cM_{L,\l} $ is a bi-Lipschitz homeomorphism (where $\cM_{L,\l}:=\pi (\cMt_{L,\l})$) and
\begin{equation}
\cMt_{L,\l} = \left\{  \left(x, \frac{\partial H}{\partial p}(x, d\u (x)\right):\quad x\in \cM_{H,\l} \right\}.
\end{equation}
\end{itemize}

\vspace{10 pt}


\section{Limit  to the conservative case}\label{sec:conv}
In this section we would like to briefly discuss what happens in the limit as the dissipation $\l$ goes to zero.

Let us start with the following property whose proof follows, for example, from \cite{LPV}, \cite[Proposition 2.6]{DFIZ} and Remark \ref{remarkconstant} ({\it i}). We denote by $\a$(0) the value of Mather's $\a$-function at $0$ ( we refer for example to \cite{Mather91, Fathibook, Sorrentinobook} for more details).

\begin{proposition}\label{convergencetoalpha}
$\l \u$ converges uniformly to $-\a(0)$ as $\l \rightarrow 0^+$.
\end{proposition}

\begin{remark}
It follows from this fact that the region $\cZ_{F_\l}^0\cup \cZ_{F_\l}^-= \{(x,p)\in T^*M :\: {F_\l}(x,p)\leq 0\}$ where the dynamics is attracted, in the limit as $\l \to 0^+$ converges to the energy sublevel $\{H(x,p)\leq \a(0)\}$.  In particular, $\cZ_{F_\l}^0$ converges to {\it Ma\~n\'e's critical energy level} for $H$
(see \cite{Fathibook, Sorrentinobook} and references therein).\\
\end{remark}

\noindent Let us now prove this convergence result.

\begin{proposition}\label{convMathmeas}
Let $\mu_\l$ be minimizing measures for $\l>0$ and assume that $\bar{\mu}$ is an accumulation point of these probability measures as $\l$ goes to $0$. 
Then, $\bar{\mu}$ is a Mather measure for the limit conservative system.
\end{proposition}

\begin{proof}
Let assume that $\mu_{\l_n}$ converge (in the weak$^*$ topology) to $\bar{\mu}$ ($\l_{n} \to 0^+$ as $n\rightarrow +\infty$). Then, it follows from the definition of minimizing measure, the convergence of these measures and 
Proposition \ref{convergencetoalpha}, that:
\begin{eqnarray*}
0&=& \int_{TM} (L- \l_n \bar{u}_{\l_n})\, d\mu_{\l_n} \stackrel{n\rightarrow +\infty}{\longrightarrow}
\int_{TM} L\, d\bar{\mu} + \a(0),
\end{eqnarray*}
which implies
\begin{equation}\label{minimcond}
\int_{TM} L\, d\bar{\mu} = - \a(0).
\end{equation}
Observe that $\bar{\mu}$ is a closed probability measure (since it is the limit of invariant, hence closed, probability measures). It has been proven in \cite[Proposition 1.3]{Mane} (see also \cite[Theorem 31]{Bernard} for the proof of the equivalence between the definition of closed measures and holonomic measures) that a closed measure which satisfies the minimality condition in (\ref{minimcond}) is invariant and it is a Mather measure.
\end{proof}

\vspace{10 pt}

If we denote by $\cMt_{L}$ the (conservative) Mather set associated to $H$ and $L$, then the following holds. 

\begin{corollary}\label{corolconvMat}
The limit of $\cMt_{L,\l}$ is contained in $\cMt_L$. More specifically, for every neighborhood $\cU \supset \cMt_L$, the sets $\cMt_{L,\l}$ are definitely contained in $\cU$ as $\l \rightarrow 0^+.$ \\
\end{corollary}

\begin{remark}\label{convhomol}
The following reasoning and the above results can be easily adapted to the case in which the cohomology class $\eta \in H^1(M;\R)$ is different from zero.
In particular, the limit to the conservative case implies that both the cohomology class  $c_{\l} \longrightarrow 0$ and $\l\rightarrow 0^+$; more specifically, in the light of Remark \ref{remnonexact} and \eqref{modifiedHam}, we are interested in the limit of $\frac{c_\l}{\l}$.\\
One can easily consider the case in which  $c_\l= \l\, c_0$. 
In this case Proposition \ref{convergencetoalpha} reads:  $\l \bar{u}_{\l,c_\l}$ converges uniformly to $-\a(c_0)$ as $\l \rightarrow  0^+$. The proof is the same, choosing the new Hamiltonian $\tilde{H}(x,p) = H(x,c_0 + p)$ (see Remark \ref{remnonexact} and \eqref{modifiedHam}). In particular, all other proofs (given for $c=0$)  adapt similarly to this case, up to substitute the limit zero cohomology class with $c_0$.
\end{remark}

\begin{remark}\label{remconv}
(i) In \cite{DFIZ} the authors proved that $\bar{u}_{\l} + \frac{\a(0)}{\l}$ uniformly converges as $\l \to 0^+$ to a specific solution to the classical
Hamilton-Jacobi equation.\\
(ii) A similar convergence result as in Corollary \ref{corolconvMat}  does not hold in general for the Aubry set. Consider for example a vector field $X$  on a closed surface $\Sigma$ and let  $H(x,p)=\frac{1}{2}\|p\|_x^2 + \langle p, X(x)\rangle_x$ be the associated Ma\~n\'e Hamiltonian (see Example \ref{exmane}).
As we have seen in Proposition \ref{AubryLaggraph} for each $\l>0$ the Aubry set
$\cAt_{L,\l} = {\rm Graph}(X)$, so $$\lim_{\l \to 0^+} \cAt_{L,\l} =  {\rm Graph}(X).$$
On the other hand, the Aubry set $\cAt_L$ for the conservative case might be smaller (see Example \ref{exmane}). In fact, as it was proven in \cite[Theorem 1.6]{FFR}, under these assumptions the projected Aubry set $\cA_L = \pi(\cAt_L)$  corresponds to the set of chain-recurrent points for the flow of $X$ on $\Sigma$; hence, it may happen that it is only properly contained in $\Sigma$.\\
\end{remark}

\section{Examples} \label{sec:examples}
Let us discuss some illustrative examples of conformally symplectic vector fields and describe the corresponding Aubry-Mather sets.\\
\begin{example}[{\bf Integrable CS Vector Fields}]\label{exint}
Let $h: \R^n \longrightarrow \R$ be a strictly convex and superlinear $C^2$ function and consider the vector field  on $T^*\T^n $given by
\begin{eqnarray*}
\left\{
\begin{array}{l}
\dot{x} = \frac{\partial h}{\partial p}(p)\\
\dot{p} = -\l p + \eta
\end{array}
\right.
\end{eqnarray*}
where $x\in \T^n$, $p\in \R^n$, while $\l>0$ and $\eta\in \R^n$ are fixed.\\
It is easy to check that the Lagrangian submanifold $\Lambda_{\l, \eta}= \T^n \times \{\frac{\eta}{\l}\}$ is invariant and that the motion on it corresponds to a rotation with rotation vector 
$\frac{\partial h}{\partial p}(\frac{\eta}{\l})$. In particular we have that
$$
\cK_{h,\l}=\cA^*_{h,\l}=\cM^*_{h,\l}=\Lambda_{\l, \eta}.
$$
All orbits that do not lie on this invariant manifold are asymptotic to $\Lambda_{\l, \eta}$. In fact, the equation $\dot{p} = -\l p + \eta$, with initial condition $p(0)=p_0$, is easy to integrate and one obtains: 
$$
p(t) = C e^{-\l t} + \frac{\eta}{\l}
$$
where $C=C(p_0)=p_0-\frac{\eta}{\l}$ is a constant depending on the initial condition (observe that it vanishes when $p_0=\frac{\eta}{\l}$); in particular, $p(t)\longrightarrow \frac{\eta}{\l}$ as $t\rightarrow +\infty.$\\

If we consider the limit from the dissipative to the conservative case, observe that when $\lambda$ goes to zero, also $\eta$ must converge to zero (otherwise the limit system does not correspond to a Hamiltonian system on $T^*\T^n$ anymore). In particular, what really matters is the value of the limit $\frac{\l}{\eta}$ as $\lambda$ goes to zero: if this limit exists and is equal to some $c\in \R$, then  $\Lambda_{\l, \eta}$ converges to the invariant tours $\T^n\times \{c\}$.

\begin{remark}
Let $H(x,p):\T^n\times\mathbb{R}^n \longrightarrow \R$ be a Tonelli Hamiltonian of the form $H(x,p)=h(p)+\e H_1(x,p)$ where $h: \R^n \longrightarrow \R$ is a strictly convex and superlinear $C^2$ function, $\e>0$, $\l>0$ and $\eta\in \R^n$, and let us consider the {\it quasi-integrable} CS vector field given by
\begin{eqnarray*}
\left\{
\begin{array}{l}
\dot{x} = \frac{\partial h}{\partial p}(p)+\e \frac{\partial H_1}{\partial p}(x,p)\\
\dot{p} = -\e\frac{\partial H_1}{\partial x}(x,p)-\l p + \eta
\end{array}
\right.
\end{eqnarray*}
Different KAM approaches ({\it e.g.}, \cite{CCD, CCD2, Massetti, Moser67}) have been proposed to show the persistence -- under suitable assumptions and for small values of $\e$ -- of the invariant torus of the integrable case $\Lambda_{\l,\eta}$. This ``perturbed'' torus does coincide with the Aubry and the Mather 
sets that we constructed; in particular, it continues to be a local attractor \cite{CCD2}.
\end{remark}
\end{example}

\medskip

\begin{example}[{\bf The dissipative pendulum}]\label{expendulum} 

Let us consider the  mechanical system obtained by adding a dissipative force proportional to the velocity to the simple pendulum equation (what is generally called the {\it dissipative pendulum}). 
The corresponding  Hamiltonian $H$ is defined on $T^*\T = \T\times\R$, with $\T =\R/2\pi\Z$, by $H(x,p)=\frac{1}{2}p^2 - (1 - \cos x)$. The corresponding Lagrangian $L(x,v)=\frac{1}{2}p^2 + (1 - \cos x)$ is defined on $T\T=\T\times\R$. 
The associated CS vector field is:  
\begin{eqnarray}\label{pendeq}
\left\{
\begin{array}{l}
\dot{x} = p\\
\dot{p} = \sin  x-\l p 
\end{array}
\right.
\end{eqnarray}
Let us now make some observations.
\begin{itemize}
\item[i)] 
We have that $H\geq -2$ and 
\[
\frac{d}{dt}H(x(t),p(t))=-\l p(t)^2
\]
so that $H$ is a Lyapunov function. For every $c> 0$  consider the  forward invariant set 
\[
M_c =\{(x,p)\in T^*\T \: : {H(x,p) < c}  \}.
\]
Applying the LaSalle invariant principle \cite{LaSalle} on $M_c$ we have that $\Omega_\infty(M_c)$ is contained in the largest forward invariant set in $\{\dot{H} = 0  \}$. Hence,  $\Omega_\infty(T^*\T)\subseteq P_1\cup P_2$ where  $P_1 = (0,0)$ and $P_2= (\pi,0)$ are the only equilibria of the system. Moreover, since $P_1$ is a saddle, by the stable manifold theorem, there exist (exactly) two orbits approaching it for $t\rightarrow +\infty$. We can apply LaSalle principle to a neighborhood of $P_2$ to get that it is asymptotically stable. Then, we have that $\Omega_\infty(T^*\T)=  P_1\cup P_2$.
 The basins of attractions of these two equilibria are different: all of the orbits converging to $P_1$ stay on its stable manifold, while the basin of attraction of $P_2$ is the rest of $T^*\T$ (see Figure \ref{fig:pend}-(a)).\\
\item[ii)] The unique solution $\bar u_\l$ to the associated $\l$-discounted Hamilton-Jacobi equation (see (\ref{defu})) enjoys some symmetries.
In fact, first note that if $(x(t),p(t))$ is an orbit of (\ref{pendeq}), then also $(2\pi-x(t),-p(t))$ is an orbit (we slightly abuse of notation, thinking of the lifted system on the covering space $\R^2$). Hence the system (\ref{pendeq}) is invariant under the action of the involution $\cI (x,p)=(-x,-p)$ defined on $T^*\T$. Since in this case $\cL(x,v)=(x,p)$, the same holds for the corresponding Lagrangian system. Moreover, both $L$ and $H$ are invariant under the action of $\cI$.
Therefore,   if $\gamma_x$ realizes the minimum in (\ref{defu}) so does $\gamma_{1-x}=\cI\circ\cL^{-1}(\gamma_x,\dot\gamma_x)$. Hence $\bar u_\l(x)=\bar u_\l(-x)$.\\
\item[iii)] We know that for every $x\in\T$ there exists $\g_x:(-\infty,0]\longrightarrow \T$ such that $\g_x(0)=x$ and which is $(\u,L)$-calibrated (see Proposition \ref{theoremonu}); in particular, using Propositions \ref{propdiffcal1} and \ref{propdiffcal2}, and the fact that in this case $\cL_L(x,v)=(x,p)$, we have that $\u$ is differentiable in $\g_x((-\infty,0))$ and that
$\dot{\g}_x(t)=d\u (\g_x(t))$ for all $t\in (-\infty,0)$.\\
\end{itemize}
These facts and and the information on the symmetry of $\u$, are sufficient to determine, at least qualitatively, $\u$. More specifically, $\u$ is differentiable everywhere but at $x=\pi$; moreover, with reference to Figure \ref{fig:pend}-(a), the graph of  $d\u$ coincides on $[0,\pi)$ with the ``upper'' part of the unstable manifold of $P_1$, and on $(\pi,2 \pi]$ with the ``lower'' part of the unstable manifold of $P_1$.\\
Hence, $\cSt$ is the union of the closure of these two branches of separatrices. As a consequence, it follows from the definition of Aubry set \eqref{defA} that 
$$\cA_{H,\l}^*=P_1=\{(0,0)\}.$$
Moreover, there is only one invariant measure supported in $\cA_{H,\l}^*$, namely the Dirac's delta $\delta_{P_1}$; therefore (see Proposition \ref{propminimmeasure}):
$$\cM_{H,\l}^*=\cA_{H,\l}^*=P_1=\{(0,0)\}.$$
It comes from observation i) that the set $\cA_{H,\l}^*$ is not an attractor. Actually, being a saddle we can define its unstable manifold whose orbits approach the asymptotically stable point $P_2$ (cfr Remark \ref{remIntro}).

%
%
%
%

%
%
%
\begin{remark}
Let $F_{\l}(x,p)= \l \u(x) + H(x,p)$. From the symmetries of $H$ and $\u$ one deduces that
$F_\l(-x,-p)=F_\l(x,p)$ and $F_\l(x,-p)=F_\l(x,p)$.
It follows from these symmetries that $\cZ^0_{F_\l}$ is obtained by reflecting $\cL_L(\cSt)$ about the axis $x=\pi$. In particular, 
$\cZ^-_{F_\l}$ is the bounded region that it encloses  (see Figure \ref{fig:pend}-(b)).\\
Finally, we claim that the maximal attractor $\cK_{H,\l}$ is formed by the equilibria $P_1$ and $P_2$ and the unstable manifolds $W^u$ of $P_1$ (see Figure \ref{fig:pend}-(c)). \\
First,  observe that $P_1$, $P_2$ and $W^u$ are contained in $\cZ^0_{F_\l} \cup \cZ^-_{F_\l}$ and are invariant under the flow, therefore, $P_1\cup P_2\cup W^u \subseteq\cK_{H,\l}$. \\
Let us prove the other inclusion.  Consider a point $P\in \cK_{H,\l}$; since $\cK_{H,\l}$ is invariant, then the alpha-limit of $P$ is  contained in $\cK_{H,\l}\subseteq \cZ^0_{F_\l}\cup\cZ^-_{F_\l}$; in particular, since $H$ is Lyapunov function of the system -- see item i) above -- then the alpha-limit of $P$ must be contained in the set 
$\{\frac{d}{dt} H =0\} = \{p=0\}$. It follows from the equations of motion, that the only invariant sets contained in $\{p=0\}$ are $P_1$ and $P_2$.
If the alpha-limit is $P_2$ then $P\equiv P_2$, while if the alpha-limit is $P_1$ then $P\equiv P_1$ or $P\in W^u$. 
This shows that $P_1\cup P_2\cup W^u \supseteq\cK_{H,\l}$, and concludes the proof. 
\end{remark}

\begin{figure}[t!]\label{fig:pend}
    \centering
    \begin{subfigure}[]
       {
        \includegraphics[scale=0.25]{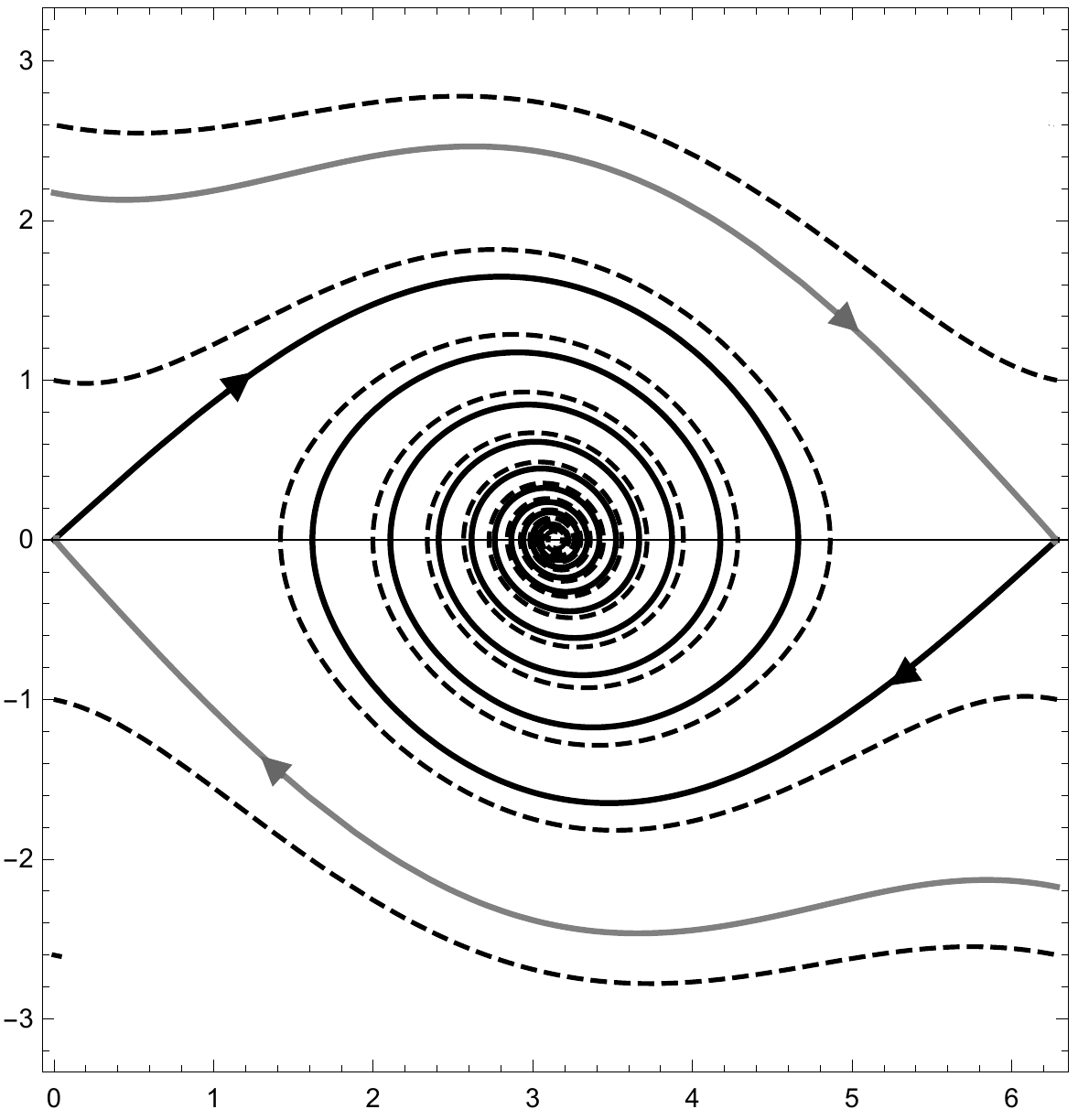}
       }
    \end{subfigure} 
    \begin{subfigure}[]
       { \includegraphics[scale=0.4]{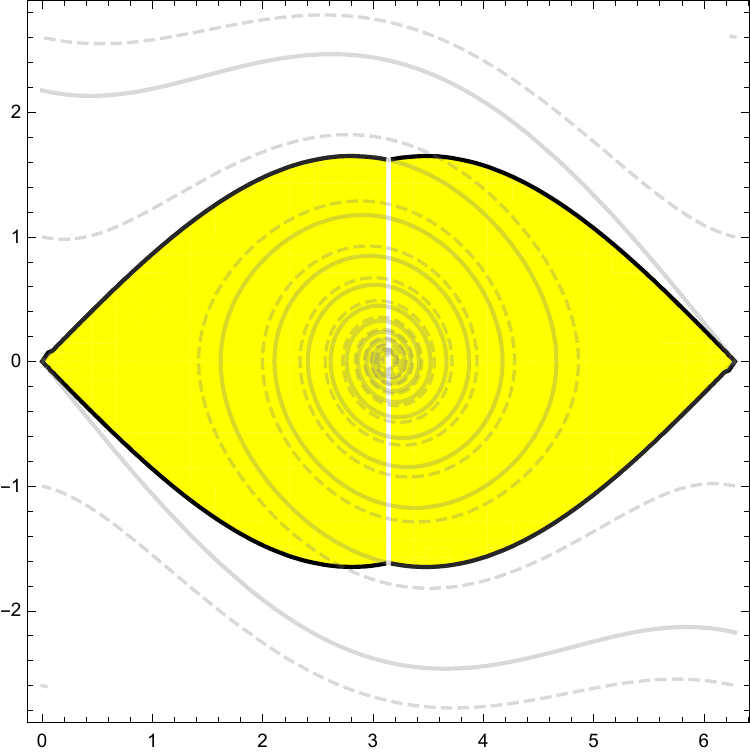} }
    \end{subfigure}\\
    \begin{subfigure}[]
   {     \includegraphics[scale=0.26]{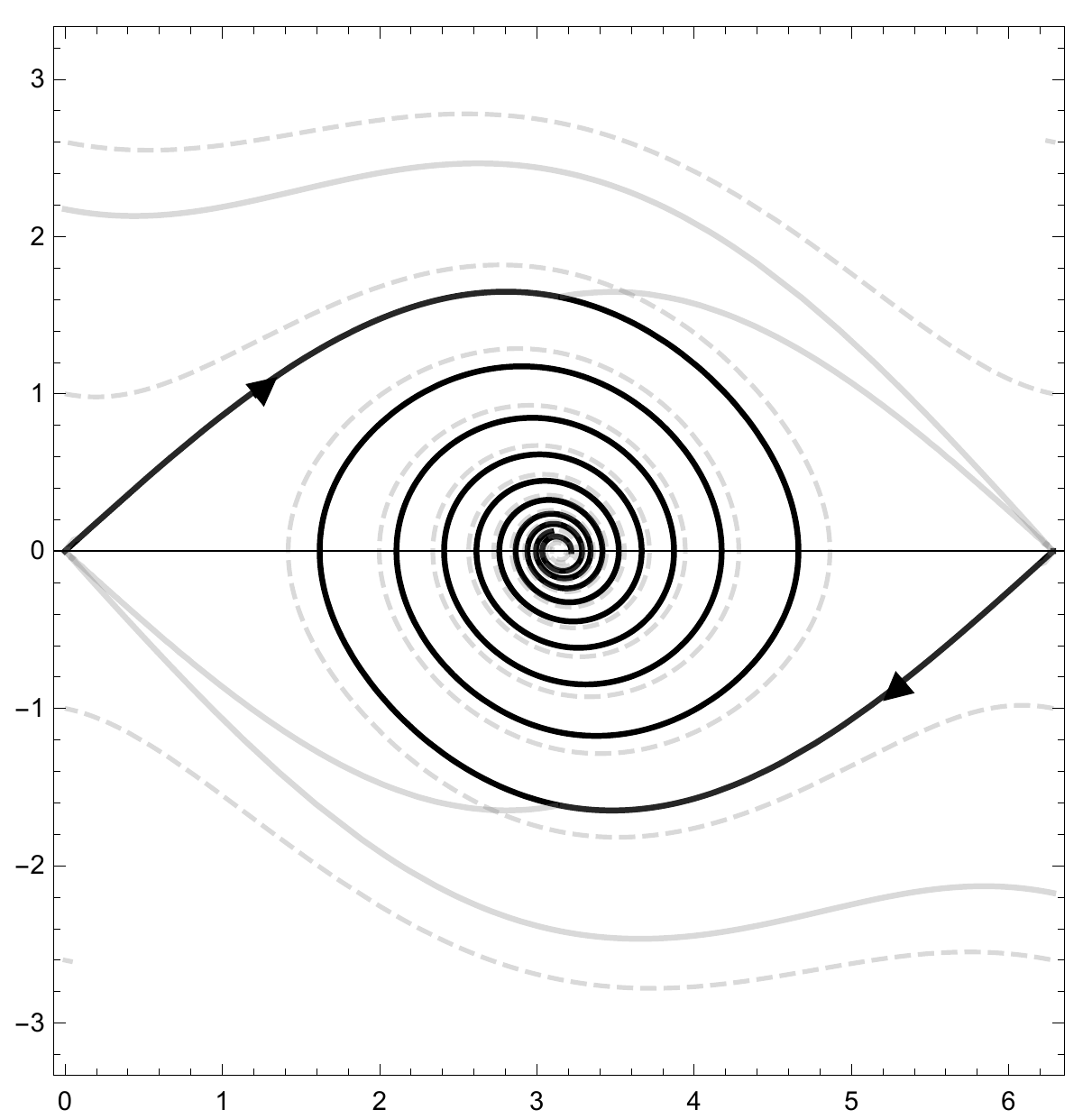} }
    \end{subfigure}
    \caption{The dissipative pendulum with $\l=1/5$. (a): Phase portrait where we highlight the stable and unstable manifolds of the saddle (thick gray and thick black respectively).(b) The sets $\cZ^0_{F_\l}$ (thick black line) and $\cZ^-_{F_\l}$ (shaded region). (c) The global maximal attractor formed by the unstable manifolds and the equilibria.}
\end{figure}
\end{example}

\bigskip

\begin{example}[{\bf Ma\~n\'e-like CS Vector Fields}]\label{exmane} 
Let $X$ be a vector field on $M$ and denote by $\varphi_X^t$ the associated flow. Consider the associated Ma\~n\'e Lagrangian $L(x,v)=\frac{1}{2}\|v-X(x)\|^2_x$. Observe that $L(x,v)\geq0$ and vanishes only on 
$${\rm Graph}(X)= \{(x,X(x))\!: \; x\in M\} \subset TM.$$
Let us consider the corresponding Hamiltonian $H(x,p)=\frac{1}{2}\|p\|_x^2 + \langle p, X(x)\rangle_x$. For every $\l>0$
we consider the vector field  on $T^*M $ defined by:
\begin{eqnarray*}
\left\{
\begin{array}{l}
\dot{x} = \frac{\partial H}{\partial p}(x,p) = p + X(x)\\
\dot{p} = -\frac{\partial H}{\partial x}(x,p) -\l p.  
\end{array}
\right.
\end{eqnarray*}
It is easy to check that the function $\u\equiv 0$ is the unique solution of (\ref{eqHJ}); hence, in the light of Proposition \ref{invariantlagrangiangraph}, we can conclude that the zero section $\cO \subset T^*M$ is invariant (clearly, it is Lagrangian and exact). In particular, the vector field restricted on it becomes 
\begin{eqnarray*}
\left\{
\begin{array}{l}
\dot{x} = X(x)\\
\dot{p} = 0.  
\end{array}
\right.
\end{eqnarray*}
Therefore, the flow $\Phi^t_{H,\l}$ on $\cO$ is smoothly conjugated to $\varphi_X^t$ (the conjugation is the projection $\pi_{|\cO}:\cO \longrightarrow M$). \\
Observe that the dynamics on this invariant manifold can be very complicated. For examples, the recurrent set might contain invariant measures with different homology (or rotation vector).
As a simple example consider the following (see also \cite[Remark 3.3.5 ({\it iv})]{Sorrentinobook}). Let $M=\T^2= \R^2/(2\pi\Z)^2$ equipped with the flat metric and consider a vector field $X$ with norm $1$ and such that $X$ has two closed orbits $\g_1$ and $\g_2$ and any other orbit asymptotically approaches $\g_1$ in forward time and $\g_2$ in backward time; for example one can consider $X(x_1,x_2)=(\cos( x_1), \sin( x_1))$, where $(x_1,x_2)\in \T^2$. Let us denote by $\tilde{\g}_1$ and $\tilde{\g}_2$ the lifts of these orbits on ${\rm Graph(X) \subset TM}$.
One can check that:
\begin{itemize}
\item It comes from proposition \ref{AubryLaggraph} that  $\cAt_{L,\l} = {\rm Graph}(X)$ 
\item The only ergodic invariant probability measures supported in $\cAt_{L,\l}$, are those supported on $\tilde{\g}_1$ and $\tilde{\g}_2$. Therefore 
$$
\cMt_{L,\l} = \tilde{\g}_1 \cup \tilde{\g}_2 \subsetneq \cAt_{L,\l}.\\
$$
\end{itemize}
\end{example}

\vspace{10 pt}



\begin{thebibliography}{10}
\expandafter\ifx\csname natexlab\endcsname\relax\def\natexlab#1{#1}\fi
\expandafter\ifx\csname bibnamefont\endcsname\relax
  \def\bibnamefont#1{#1}\fi
\expandafter\ifx\csname bibfnamefont\endcsname\relax
  \def\bibfnamefont#1{#1}\fi
\expandafter\ifx\csname citenamefont\endcsname\relax
  \def\citenamefont#1{#1}\fi
\expandafter\ifx\csname url\endcsname\relax
  \def\url#1{\texttt{#1}}\fi
\expandafter\ifx\csname urlprefix\endcsname\relax\def\urlprefix{URL }\fi
\providecommand{\bibinfo}[2]{#2}
\providecommand{\eprint}[2][]{\url{#2}}



\bibitem{Arnold}

Vladimir I. Arnol$'$d.
\newblock Mathematical methods of classical mechanics.
\newblock {\em Graduate Texts in Mathematics}, Vol. 60 (Second edition), Springer-Verlag, New York, xvi+516, 1989.



\bibitem{ABC}
Alberto Abbondandolo, Olga Bernardi and Franco Cardin.
\newblock Chain recurrence, chain transitivity, Lyapunov functions and rigidity of Lagrangian submanifolds of optical hypersurfaces.
\newblock Preprint, 2015.


\bibitem{Banyaga}
Augustin Banyaga.
\newblock Some properties of locally conformal symplectic structures.
\newblock {\em Comment. Math. Helv.} 77 (2): 383--398, 2002.


\bibitem{Barles}
Guy Barles.
\newblock Solutions de viscosit\'e des \'equations de Hamilton-Jacobi.
\newblock {\em Math\'ematiques \& Applications (Berlin) [Mathematics \& Applications]} Vol 17, Springer-Verlag, Paris, 1994.

\bibitem{Benso}
 Alain Bensoussan.
\newblock Perturbation Methods in Optimal Control.
\newblock Wiley/Gauthier-Villars Ser. Modern Appl. Math., John Wiley \& Sons Ltd., Chichester, 1988, 
\newblock translated from the French by C. Tomson.


\bibitem{Bernard}
Patrick Bernard.
\newblock  Young measures, superposition and transport.
\newblock {\em Indiana Univ. Math. J.} 57 (1): 247--275, 2008.



\bibitem{CCD}
Renato Calleja, Alessandra Celletti and  Rafael de la Llave. 
\newblock A KAM theory for conformally symplectic systems: efficient algorithms and their validation.
\newblock {\em J. Differential Equations} 255 (5): 978--1049, 2013.


\bibitem{CCD2}
Renato Calleja, Alessandra Celletti and  Rafael de la Llave. 
\newblock Local behavior near quasi-periodic solutions of conformally symplectic systems. 
\newblock {\em J. Dynam. Differential Equations}  25  (3): 821--841, 2013.


\bibitem{Casdagli}
Martin Casdagli.
\newblock Periodic orbits for dissipative twist maps.
\newblock {\em Ergodic Theory Dynam. Systems} 7 (2): 165-173, 1987. 

\bibitem{Cellettibook}
Alessandra Celletti.
\newblock Stability and Chaos in Celestial Mechanics.
\newblock Springer-Verlag, Berlin, 2010 (published in association with Praxis
Publ. Ltd., Chichester).



\bibitem{Conley}
Charles Conley.
\newblock The gradient structure of a flow (I).
\newblock {\em Ergodic Theory Dynam. Systems} 8: 11--31, 1988.


\bibitem{DFIZ}
Andrea Davini, Albert Fathi, Renato Iturriaga and Maxime Zavidovique.
\newblock Convergence of the Solutions of the Discounted  Hamilton-Jacobi Equation.
\newblock to appear in {\em Inv. Math.}, 2014.

\bibitem{DM}
Carl P. Dettmann and G.P. Morris.
\newblock Proof of Lyapunov exponent pairing for systems at constant kinetic energy.
\newblock {\em Phys. Rev. E}  53 (6): R5545--R5548, 2006.


\bibitem{FFR}
Albert Fathi, Alessio Figalli and Ludovic Rifford.
\newblock On the Hausdorff Dimension of the Mather Quotient.
\newblock {\em Comm. Pure Appl. Math.} 62 (4): 445--500, 2009.


\bibitem{Fathibook}
Albert Fathi.
\newblock Weak KAM Theorem in Lagrangian Dynamics.
\newblock Preprint (10th preliminary version), 2008.

\bibitem{IS}
Renato Iturriaga and Hector Sanchez-Morgado.
\newblock Limit of the infinite horizon discounted Hamilton-Jacobi equation.
\newblock {\em Discrete Contin. Dyn. Syst. Ser. B} 15 (3): 623--635, 2011.


\bibitem{LeCalvez}
Patrice Le Calvez.
\newblock Existence d'orbites quasi-p\'eriodiques dans les attracteurs de Birkhoff.
\newblock {\em Comm. Math. Phys.} 106 (30):  383--394, 1986.

\bibitem{LaSalle}
Joseph P. LaSalle.
\newblock Stability Theory for Ordinary Differential Equations.
\newblock {\em J. Differential Equations} 4 (1): 57--65, 1968.


\bibitem{LeCalvez2}
Patrice Le Calvez.
\newblock Propri\'et\'es des attracteurs de Birkhoff.
\newblock {\em Ergodic Theory Dynam. Systems} 8 (2): 241--310, 1988.



\bibitem{LeCalvezbook}
Patrice Le Calvez.
\newblock Dynamical properties of diffeomorphisms of the annulus and of the torus. 
\newblock {\em SMF/AMS Texts and Monographs,} 4. American Mathematical Society, Providence, RI; SociŽtŽ MathŽmatique de France, Paris,  x+105 pp., 2000.

\bibitem{LPV}
Pierre-Louis Lions, Georgios Papanicolaou and S. R. Srinivasa Varadhan. 
\newblock Homogeneization of Hamilton-Jacobi equations. 
\newblock Preprint, 1987.


\bibitem{WL}
Carlangelo Liverani and Maciej P. Wojtkowski.
\newblock Conformally symplectic dynamics and symmetry of the Lyapunov spectrum.
\newblock {\em Comm. Math. Phys.} 194 (1): 47--60, 1998.


\bibitem{Mane}
Ricardo Ma\~n\'e.
\newblock Generic properties and problems of minimizing measures of Lagrangian systems.
\newblock {\em Nonlinearity} 9: 273--310, 1996.

\bibitem{Massetti}
Jessica E. Massetti.
\newblock  Normal forms for perturbations of systems possessing a diophantine invariant torus.
\newblock Preprint 2016, available at http://arxiv.org/pdf/1511.02733.pdf 

\bibitem{Mather91}
John N. Mather.
\newblock Action minimizing invariant measures for positive definite Lagrangian systems.
\newblock {\em Math. Z.} 207 (2): 169--207, 1991.

\bibitem{milnor1}
Milnor, J.
\newblock On the concept of attractor.
\newblock {\em Commum. Math. Phys.}  99: 177--195, 1985.

\bibitem{milnor2}
Milnor, J.
\newblock On the concept of attractor: Correction and remarks.
\newblock {\em Commum. Math. Phys.} 102: 517--519, 1985.


\bibitem{Moser67}
J\"urgen Moser.
\newblock Convergent series expansions for quasi-periodic motions.
\newblock {\em Math. Annalen} 169 (2): 136--176, 1967.

\bibitem{Sorrentinobook}
Alfonso Sorrentino.
\newblock {Action-minimizing methods in Hamiltonian dynamics: an introduction
to Aubry-Mather theory.}
\newblock {\em Mathematical Notes} Vol. 50, Princeton University Press, 128 pp., 2015.


\bibitem{Vaisman}
Izu Vaisman.
\newblock On locally conformal almost K\"ahler manifolds.
\newblock {\em  Israel J. Math.} 24 (3-4): 338--351, 1976.


\bibitem{Vaisman2}
Izu Vaisman.
\newblock Locally conformal symplectic manifolds.
\newblock {\em Internat. J. Math. Math. Sci.} 8 (3): 521--536, 1985.
 


\end{thebibliography}
\end{document}